\documentclass[twoside,11pt]{article}

\usepackage[preprint]{jmlr2e}
\usepackage{amsmath}
\usepackage[nameinlink]{cleveref}
\usepackage{comment}%PAra hacer comentarios largos.
\usepackage{algorithm}%Para los algoritmos
\usepackage[noend]{algpseudocode}
\usepackage{makecell}
\newcommand{\weak}{\rightharpoonup}

\newcommand{\be}{\begin{equation}}
\newcommand{\ee}{\end{equation}}
\newcommand{\ba}{\begin{eqnarray}}
\newcommand{\ea}{\end{eqnarray}}
\newcommand{\beq}{\begin{equation}}
\newcommand{\eeq}{\end{equation}}

\renewcommand{\leq}{\leqslant}
\renewcommand{\le}{\leqslant}
\renewcommand{\geq}{\geqslant}
\renewcommand{\ge}{\geqslant}
\providecommand{\norm}[1]{\left\lVert#1\right\rVert}

\renewcommand{\S}{\mathcal S}
\def \R {\mathbb{R}}
\def \P {\mathcal{P}}
\def \S {\mathcal{S}}
\def \B {\mathcal{B}}
\def \N {\mathbb{N}}
\def \Om {\Omega}

\def \dis {\displaystyle}

\def\beq{\begin{equation}}
\def\eeq{\end{equation}}
\def\ecart{\noalign{\medskip}}
\def\ba{\begin{array}}
\def\ea{\end{array}}

\newtheorem{theo}{Theorem}%[section]
\newtheorem{prop}{Proposition}%[section]
\newtheorem{defi}{Definition}%[section]

\newtheorem{rem}[defi]{Remark}

\DeclareMathOperator*{\argmin}{argmin}
\DeclareMathOperator*{\argmax}{argmax}
\jmlrheading{1}{2000}{1-48}{4/00}{10/00}{meila00a}{Karl Kunish, Donato Vásquez-Varas and Daniel Walter}

\ShortHeadings{Learning Optimal Feedback Operators and their Polynomial Approximation}{Kunisch, Vásquez-Varas and Walter}
\firstpageno{1}

\begin{document}

\title{Learning Optimal Feedback Operators and their Polynomial Approximation}

\author{\name Karl Kunisch \email karl.kunisch@uni-graz.at \\
       \addr Radon Institute for Computational and Applied Mathematics\\
Austrian Academy of Sciences\\
and\\
Institute of Mathematics and Scientific Computing\\
University of Graz\\
Heinrichstraße 36, A-8010 Graz, Austria
       \AND
       \name Donato Vásquez-Varas \email donato.vasquez-varas@ricam.oeaw.ac.at\\ \addr Radon Institute for Computational and Applied Mathematics\\
Austrian Academy of Sciences\\
Altenbergerstraße 69, A-4040 Linz, Austria
\AND
\AND
       \name Daniel Walter \email daniel.walter@oeaw.ac.at\\ \addr Radon Institute for Computational and Applied Mathematics\\
Austrian Academy of Sciences\\
Altenbergerstraße 69, A-4040 Linz, Austria}

\editor{}

\maketitle

\begin{abstract}%   <- trailing '%' for backward compatibility of .sty file
A learning based  method for obtaining feedback laws for nonlinear optimal control problems is proposed. The learning problem is posed such that the open loop value function is its optimal solution.
 This infinite dimensional, function space, problem, is approximated by a polynomial ansatz and its convergence is analyzed. An $\ell_1$ penalty term is employed, which combined with the proximal point method, allows to find sparse solutions for the learning problem. The approach requires multiple evaluations  of the elements of the polynomial basis and of their derivatives.  In order to do this efficiently a graph-theoretic algorithm is devised.  Several examples underline that the proposed methodology provides a promising  approach for  mitigating the curse of dimensionality which would be involved in case the optimal feedback law was obtained by solving  the Hamilton Jacobi Bellman equation.
\end{abstract}

\begin{keywords}
 Optimal feedback control, nonlinear systems, learning theory,  Hamilton Jacobi Bellman equation, polynomial based approximation.
\end{keywords}

\section{Introduction}
Designing optimal feedback laws for non-linear control problems is a challenging  problem from both the theoretical and applied points of view. The main approach for obtaining  an optimal feedback law
 is based on dynamical programming. Its solution involves the control theoretic
  Hamilton-Jacobi-Bellman (HJB) equation.  For high dimensional problems the computational cost of directly  solving the HJB equation makes this approach  non-viable. In the last years many efforts have been put forward to overcome this difficulty and to partially alleviate the curse of dimensionality. Here we can only mention a very small sample of the large number of contributions: representation formulas \citep{Chow1,Chow2,Chow3,DarbonOsher}, approximating the HJB equation by neuronal networks \citep{Han,Darbon,Nusken,Onken,Ito,KunischWalter,Ruthotto}, data driven approaches \citep{Nakamura1,Nakamura2,AzKaKK,Kang,Albi},  max-plus methods \citep{Akian,Gaubert,Dower}, polynomial approximation \citep{Kalise1,Kalise2}, tensor decomposition
methods \citep{Horowitz,Stefansson,Gorodetsky,Dolgov,Oster,Oster2}, POD methods \citep{Alla2,KunischVolk}, tree structure algorithms \citep{Alla1}, and sparse grids  techniques\citep{BokanowskiGarckeGriebelPo, Garcke, KangWilcox}, see also the proceedings volume \citep{KaliseKuRa}. Among the classical methods for solving the HJB equation we mention finite difference schemes \citep{Bonnans}, semi-Lagrangian schemes \citep{Falcone}, and policy iteration \citep{Alla, Beard, Puterman, Santos}.

\par
 In the present  work we propose, analyze, and numerically test a learning approach to obtain optimal feedbacks laws. The problem is formulated in a way that all its solutions are optimal feedback laws.
 Due to the infinite dimensional nature of the problem, a  finite dimensional approximation is required. We propose polynomials as  ansatz functions and add an $\ell_1$ penalty term, in order to promote sparsity in the solutions. Appropriate hypotheses on the value function are provided which  ensure the existence of a solution to this problem. Furthermore,  convergence is established as the dimension of the ansatz space tends to infinity.  In order to efficiently evaluate elements of the chosen polynomial basis and their first and second order derivatives, a tree-based procedure is devised.

 This work is an extension of the developments  commenced in \citep{KunischWalter}. Differently from the present paper in \citep{KunischWalter} the learning problem is formulated for the gradient of the value function and the approximation is based on neural networks. Here we learn the value function itself and utilize its gradient in the feedback law. The choice of monomials  as ansatz functions  turned out to be computationally very promising.  Certainly it would also be of interest  to investigate other non-grid based, approximation schemes in the future.
\par
The structure of this work is as follows. In \Cref{StatementOfProblem} we introduce the learning problem and in \Cref{PolyLearningSection} we present  the finite dimensional approximation by polynomials. In \Cref{ExistenceConvergenceSection} the existence of solutions and convergence of the finite dimensional problems are established. The optimality conditions for the finite dimensional problem are studied in \Cref{OptimalityConditionsSection}, together with a basis reduction procedure. The learning algorithm is developed in \Cref{OptimizationAlgoSection}. An efficient polynomial basis evaluation method is developed in \Cref{EvaluationSection}. In \Cref{GeneralizationSection} we present a result concerning the generalization capability of our approach. Finally, in \Cref{NumericalExperimentsSection} we present four numerical experiments which show that our algorithm is able to solve non-linear and high (here the dimension is 40) dimensional control problems in a standard laptop environment.

\par We end the section by introducing some notation which is needed in the following. For $k,m,d$ all integers greater than or equal 1,  and a domain $A\subset\R^{m}$, the spaces $H^{k}(A;\R^{d})$ and $H_{loc}^{k}(A;\R^{d})$ denote the Sobolev spaces and local Sobolev spaces of order $k$ from $A$ to $\R^{d}$. Analogously, for $p\geq 1$ the spaces  $L^{p}(A;\R^{d})$ and $L_{loc}^{p}(A;\R^{d})$ are the spaces of $p$ integrable and locally  $p$  integrable functions from $A$ to $\R^{d}$. In addition, for a compact set $K\subset\R^{m}$ and $\alpha\in (0,1]$, we define  $C^{k,\alpha}(K;\R^{d})$ to be the class of  $k$ times differentiable functions with $\alpha-$Hölder continuous derivatives up to order $k$  from $K$ to $\R^{d}$. For a Lipschitz continuous function $v:K\to \R^{m}$ we define
$$ |v|_{Lip(K)}=\max_{x\neq y\in K}\frac{|v(x)-v(y)|}{|x-y|},  $$
where $|\cdot|$ is the usual euclidean norm. For $y\in\R^{m}$ we denote the $p-$norm with $p\in (1,\infty)\setminus\{2\}$ by $|y|_{p}$, and the maximum norm by $|y|_{\infty}$.
\section{Statement of the Problem}
\label{StatementOfProblem}
In this work we study the  infinite horizon optimal control problem:
\beq
\dis\min_{\quad u\in L^{2}((0,\infty);\R^{m})} J(u,y_0):= \int_{0}^{\infty}\ell(y(t))dt+\frac{\beta}{2} \int_{0}^{\infty}|u(t)|^{2}dt
\label{ControlProblem}
\eeq
where $y\in {H_{loc}^{1}((0,\infty);\R^{d})}$ is the unique solution of
\beq
 y'(t)=f(y(t))+Bu(t),\quad t\in (0,\infty),\quad y(0)=y_0.
 \label{Ode1}
\eeq
Here $f:\R^{d}\to \R^{d}$ is Lipschitz on bounded sets, $\ell:\R^{d}\to \R$ is of class $C^{1}$, bounded below by 0, $\beta>0$ is the penalization for the control, and $B\in \R^{d\times m}$ is the control matrix, with $d\geq m\in \N$. We also assume that $\ell(0)=0$ and $f(0)=0$. This implies that $0$ is an equilibrium for system \eqref{Ode1}. {In \eqref{ControlProblem} the cost $J$ is considered as extended real-valued function.} \par
Problem \eqref{ControlProblem} can be solved by means of  dynamical programming. Namely, defining the value function of \eqref{ControlProblem} by
\begin{equation}
V(y_0)=\min_{\quad u\in L^{2}((0,\infty);\R^{m})} J(u,y_0),
\end{equation}
and assuming that $V$ is differentiable in an open neighborhood $U\subset\R^{d}$ of $y_0$, then $V$ solves the Hamilton Jacobi Bellman equation
\beq
\min_{u\in \R^{m}}\left\{\nabla V(y)^{\top}\left(f(y)+Bu\right)+\frac{\beta}{2}|u|^2+\ell(y)\right\}=0
\label{HJB}
\eeq
in $U$. By the verification's Theorem, the optimal control in \eqref{ControlProblem} is given by the feedback law:
\begin{equation}
    u^{*}(t)=-\frac{1}{\beta}B^{\top}\nabla V(y^{*}(t)),
    \label{HJBFeedbackLaw}
\end{equation}
provided $y^{*}(t)\in U$. Here $y^{*}$ is the solution of \eqref{Ode1} corresponding to $u^{*}$. Replacing $u$ by $u^*$ in \eqref{Ode1} we get the closed loop system \begin{equation}
    y'(t)=f(y(t))-\frac{1}{\beta}BB^{\top}\nabla V(y(t)),\quad t\in (0,\infty),\quad y(0)=y_0.
    \label{ClosedLoopProblem}
\end{equation}
This approach involves solving the Hamilton Jacobi Bellman equation, which is computationally expensive or even unfeasible  for problems of  high dimension. Therefore, in this work we propose to find a feedback law by solving a learning problem. For this purpose we define a computational domain
$\Omega=(-l,l)^{d}$ with $l>0$, and choose a set of initial conditions $\{y_{0}^{i}\}_{i=0}^{I}\subset\Omega$.
With these quantities we formulate the problem
\begin{equation}
    \min_{v\in C^{1,1}(\overline{\Om}),\ \nabla v(0)=0,\ v(0)=0}\quad
    \frac{1}{I}\sum_{i=1}^{I}\int_{0}^{\infty}\ell(y_{i}(t))dt+\frac{1}{2\beta}\int_{0}^{\infty}|B^{\top}\nabla v (y_{i}(t))|^{2},
    \label{LearningProblem}
\end{equation}
where  $y_{i} \in H_{loc}^{1}((0,\infty),\R^{d})$ for $i=1,\ldots,I$ are the solutions of the following closed loop problems
\begin{equation}
\begin{array}{l}
  \dis y_{i}'(t)=f(y_{i}(t))-\frac{1}{\beta}BB^{\top}\nabla v(y_{i}(t)),\ y_i(0)=y_0^{i},\  y_{i}(t)\in \overline{\Omega},\ t \in (0,\infty).
    \end{array}
    \label{ClosedLoopProblem2}
\end{equation}
\par
Conditions that we shall impose on $\Omega$ and $V$  below will guarantee that $V$ is a solution of \eqref{LearningProblem}. In order to solve this problem numerically we replace the infinite dimensional function space  $C^{1,1}(\overline{\Omega})$ by a finite dimensional space. In this case we add a penalty term in order to ensure the existence of at least one solution. Moreover for numerical purposes we introduce a finite horizon formulation. This problem will be formulated in the following section, where we also state  results regarding the existence of solutions and their convergence to a solution of \eqref{ClosedLoopProblem2}.
\section{Polynomial Learning Problem}
\label{PolyLearningSection}
\par In this section we formulate the finite dimensional learning problem. For this purpose we first introduce some notation. Let  $n\in \N$ and $d\in \N$, where $\N=\{0,1,2,\ldots\}$. We denote the space of polynomials with total degree less than or equal to $n$ in $\R^{d}$ by $\P_{n}$ and its dimension by $m_{n}$. For a multi-index $\alpha=(\alpha_1,\ldots,\alpha_d)\in \N^{d}$ we define a monomial $\phi_{\alpha}$ by
\beq \phi_\alpha(y)=\prod_{j=1}^{d}y_{j}^{\alpha_j},\quad y \in \R^{d}.\label{Monomial}\eeq
We denote by $\Lambda_{n}$ the set of multi-indexes such the sum of all its elements is lower or equal than $n$, that is
\beq\Lambda_n=\left\{ \alpha\in \N^{d}:\ \sum_{j=1}^{d}\alpha_j\leq n\right\} \label{AlphaIndexes}\eeq
We assume that the set of all the multi-index $\N^{d}$ is ordered such that
\beq \N^{d}=\Big\{\alpha^{i}\Big\}_{i=1}^{\infty}\mbox{ and }  \Lambda_{n+1}=\Lambda_n\bigcup{\Big\{\alpha^i\Big\}_{i=m_{n}+1}^{m_{(n+1)}}}, \label{PolyBasisInclusion}\eeq
for example, for $d=2$, we have $\Lambda_1=\{(0,0);(0,1);(1,0)\}. $
We denote the set of all the monomials with total degree lower or equal to $n$ by $\B_{n}$. Therefore, by \eqref{AlphaIndexes} and \eqref{PolyBasisInclusion} we have
\beq \B_{n}=\{\phi_{\alpha}:\alpha\in \Lambda_n\} \mbox{ and } \B_{n+1}=\B_n\bigcup{\Big\{\phi_{\alpha^i}\Big\}_{i=m_{n}+1}^{m_{(n+1)}}},\eeq

We denote the {\em hyperbolic cross} multi-index set by $\Gamma_{n}$, i.e.
\beq
\Gamma_{n}=\Big\{\alpha=(\alpha_1,\ldots,\alpha_d)\in \N^{d}:\quad \prod_{j=1}^{d}(\alpha_j+1)\leq n+1\Big\}.
\label{HyperCrossMultiindex}
\eeq
\par
We also introduce the subset $\mathcal{S}_{n}$ of $\B_n$ composed by the elements of $\B_n$ associated to the multi-indexes in $\Gamma_n$, i.e.
\begin{equation}
    \mathcal{S}_{n}=\{\phi_{\alpha}:\ \alpha\in\Gamma_{n}  \}
\end{equation}
and the subspace $\mathcal{A}_{n}$ of $\P_n$ generated by $\mathcal{S}_{n}$. \par

It is important to observe that the cardinality of $\Lambda_n$ is $\sum_{j=1}^{n}{d+j+1 \choose j}$, on the other hand the cardinality of $\Gamma_n$ is bounded by $\min\{2n^{3}4^{d},e^{2}n^{2+\log_2(n)}\}$ \citep{Adcock2017}. Hence,  for high $d$ the cardinality of the hyperbolic cross is smaller than cardinality of $\Lambda_n$. For this reason, $\Lambda_n$ is more suitable for high dimensional problems. \par

Let us consider a set of initial conditions  $\{y_{0}^{1},\ldots,y_{0}^{I}\}\subset\R^{d}$, which will be called the training set. We assume throughout that the value function $V$ is of class $C^{1,1}(\overline{\Omega})$ and  that the image of the solutions of the closed loop problems \eqref{ClosedLoopProblem} are strictly contained in $\Omega$, that is, there exists $\delta\in (0,l)$ such that
\beq |y_{i}(t)|_{\infty}\leq l-\delta \mbox{ for all }t\in (0,\infty), \label{LInftyBound}\eeq
where $y_i$ for $i\in\{1,\ldots,I\}$ are the solutions to the closed loop problem \eqref{ClosedLoopProblem} for $y_0=y_0^{i}$. Under these conditions  $V$ is  a solution of \eqref{LearningProblem}.

 For $T\in(0,\infty)$ we define $\mathcal{J}_{T}:C^{1}(\overline{\Omega})\to [0,\infty]$ by
\beq \label{eq:Jdef}
v\mapsto \mathcal{J}_{T}(v)=\frac{1}{I}\sum_{i=1}^{I}\int_{0}^{T}\left(\ell(y_{i}(t))+\frac{1}{2\beta}|B^{\top}\nabla v(y_{i}(t))|^{2}\right)dt,
\eeq
where $y_{i}$ are the solutions of the closed loop problems
\begin{equation}
\left\{\begin{array}{l}
  \dis y_{i}'(t)=f(y_{i}(t))-\frac{1}{\beta}BB^{\top}\nabla v(y_{i}(t)),\ y_i(0)=y_0^{i},\\ \ecart \dis
  y_{i}(t)\in \overline{\Omega},\ t \in (0,T).
    \end{array}\right.
    \label{ClosedLoopProblem3}
\end{equation}
If there exists an index $i\in \{1,\ldots,I\}$ such that problem \eqref{ClosedLoopProblem3} has no solution for a given $v\in C^{1,1}(\overline{\Omega})$, then we set $\mathcal{J}_{T}(v)=\infty$. Further, we define $\mathcal{J}_{\infty}:C^{1}(\overline{\Omega})\to [0,\infty]$ as the pointwise limit of $\mathcal{J}_{T}$ when $T$ goes to infinity, i.e.
\beq
v\mapsto \mathcal{J}_{\infty}(v):=\lim_{T\to\infty}J_{T}(v).
\eeq

To introduce a family of approximating computationally tractable  problems, we
consider a finite set of monomials $X=\{\phi_{i}\}_{i=1}^{M}$, which can be $\B_n\setminus\B_1$ or $S_n\setminus\B_1$. Here we subtract $\B_1$ to ensure $v(0)=0$ and $\nabla v(0)=0$. Then, for $\theta\in \R^{M}$, setting
$$v=\sum_{i=1}^{M}\theta_i \phi_i,$$ we define
\beq  \tilde{\mathcal{J}}_{T}(\theta)=\mathcal{J}_{T}\left(v\right).
\eeq
Further we define   a penalty function $P_{\gamma,r}$, with $\gamma>0$ and $r\in [0,1]$  by
\begin{equation}
    \theta\mapsto P_{\gamma,r}(\theta)=\gamma\left(\frac{(1-r)}{2}|\theta|_2^2 + r|\theta|_1\right).\label{PenaltyTerm}
\end{equation}

Now we are in a  position to introduce the finite dimensional version of the learning problem. That is, we replace $C^{1,1}(\overline{\Omega})$ with the space of polynomial spanned by $X$, add the penalty $P_{\gamma,r}$ to the objective function and consider a finite time horizon $T\in (0,\infty)$, namely
\beq \min_{\theta\in \R^{M}}  \tilde{\mathcal{J}}_{T}(\theta)+P_{\gamma,r}(\theta)
    \label{PolyLearningProblem}
\eeq
The penalty term $P_{\gamma,r}$  ensures the coercivity of the objective function. Moreover the non-smooth term in \eqref{PenaltyTerm} promotes the sparsity of the solution of \eqref{PolyLearningProblem}. However, unless we assume some further hypotheses on the structure of $\ell$, $f$, $B$ and/or the value function $V$, we do not yet know if there exist a solution of \eqref{PolyLearningProblem}.

\section{Existence and Convergence}
\label{ExistenceConvergenceSection}
In this section we are concerned with the existence of solutions for \eqref{PolyLearningProblem}.
For $n \ge 2$, $T\in (0,\infty]$, $X=\B_n\setminus \B_1$ or $X=\S_n\setminus \B_1$, and $M$ the cardinality of $X$, we say that $\theta\in \R^{M}$ is a feasible solution for problem \eqref{PolyLearningProblem} if $\tilde{\mathcal{J}}_{T}(\theta)<\infty.$ If there exits a feasible solution for problem \eqref{PolyLearningProblem}, we say that the problem is feasible.
\begin{theo}
Consider $\gamma>0$, $r\in [0,1]$, $T\in(0,\infty]$ and $X=\{\phi_i\}_{i=1}^{M}\subset C^{1,1}(\overline{\Omega})$. If problem \eqref{PolyLearningProblem} is feasible,
then it has at least one solution.
\label{ExistenceTheo}
\end{theo}
\par

The proof of this theorem as well as of the remaining results of this section are given in Section \ref{Apendice}.

In general, we do not know if there exists any feasible solution for the learning problem. Nevertheless, given that $V$ is in $C^{1,1}(\overline{\Omega})$ and the density of the polynomials in $C^{1,1}(\overline{\Omega})$, we prove that for every finite time horizon $T>0$ there exists a degree high enough, such that \eqref{PolyLearningProblem} has at least one feasible solution. Moreover, assuming the exponential stability of the closed loop problem \eqref{ClosedLoopProblem} and that $V\in C^{2}(\overline{\Omega})$, we obtain that there exists a feasible solution.

\begin{prop}
\label{FeasibleSolProp}
For every $T>0$, there exits a positive integer $n$ and $\tilde{V}\in \P_{n}$ such that $\mathcal{J}_{T}(\tilde V)<\infty.$ Moreover, assuming that the value function $V$ is $C^{2}(\overline{\Omega})$, that
\beq \lim_{t\to \infty}y_i(t)=0, \mbox{ for all }i\in\{1,\ldots,I\},\eeq
where $\{y_i\}_{i=1}^{I}$ are the solutions of the closed loop problems \eqref{ClosedLoopProblem}, and  that the linearized system
\beq  \label{eq4.2}
z'=\left(Df(0)-\frac{1}{\beta}BB^{\top}\nabla \tilde{V}(0)\right)z,\quad z(0)=z_0
 \eeq
is exponentially stable,
we have that $\mathcal{J}_{\infty}(\tilde{V})<\infty.$
\end{prop}
\par
Above we call system \eqref{eq4.2} exponentially stable if  there exist $C>0$ and $\mu>0$ such that
$|z|\leq Ce^{-\mu t}|z_0|\mbox{ for all } t \in (0,\infty)\mbox{ and }z_0\in \R^{d}$. \Cref{FeasibleSolProp} is a direct consequence of Theorem 9 in \cite[Section 7.2]{Hayek}, \Cref{ExistenceProp} and \Cref{TimeExistenceLemma}, which can be found in \Cref{Apendice}. We now address the convergence of problem \eqref{PolyLearningProblem} to \eqref{LearningProblem}.

\begin{comment}
\begin{theo}
Assume \eqref{DensityHyp} holds. Then there exist sequences $n_k\in \N$, $\gamma_k>0$, $r_k\in [0,1]$ and $\theta^{k}\in \R^{M_k}$ of solution of \eqref{PolyLearningProblem} for $X=\B_{n_k}$ or $X=\S_{n_k}$ with $M_k$ the cardinality of $X$, such that $n_{k}\to \infty$, $\gamma_k\to 0$ and $\tilde{\mathcal{J}}_{\infty}(\theta_{k})$ converges to the value of \eqref{ControlProblem} when $k\to\infty$. Moreover, denoting $v_k=\sum_{i=1}^{M_k}\theta_i^k\phi_{i}$ where $X=\{\phi_i\}_{i=1}^{M_k}$ and defining \beq u_{i}^{k}=-\frac{1}{\beta}B^{T}\nabla v_{k}(y_{i}^{k}),\eeq where $y_{i}^{k}$ is the solution of the closed loop problem \eqref{ClosedLoopProblem3} for $v=v_{k}$, $T=\infty$ and $y_0=y_0^i$, we have that
\begin{equation}
    y_{i}^{k}\weak y^{*}_{i} \mbox{ in }H^{1}_{loc}((0;\infty);\R^{d})\mbox{ and } u_{i}^{k}\weak u^{*}_i \mbox{ in }L_{loc}^{2}((0;\infty);\R^{m}),\ \mbox{when }k\to\infty,
\end{equation}
where $u_{i}^{*}$ is a solution of the open loop problem \eqref{ControlProblem} and $y_{i}^{*}$ the solution of \eqref{Ode1} for $y_0=y_0^i$.
\label{ConvergenceTheo1}
\end{theo}
\end{comment}
\begin{theo}
There exist sequences $T_{k}\in (0,\infty)$, $\gamma_{k}\in (0,\infty)$, $r_{k}\in [0,1]$, $X_k=\B_{n_{k}}\setminus \B_1$ respectively $X_k=\S_{n_k}\setminus \B_1$ with $n_{k}\in \N$ and $M_k$ the cardinality of $X_k$, and $\theta^{k}\in \R^{M_k}$ solution of \eqref{PolyLearningProblem},  such that $n_{k}\to\infty$, $T_k\to \infty$, $\gamma_{k}\to 0$ and $\tilde{\mathcal{J}}_{T_{k}}(\theta_k)$  converges to the value of \eqref{LearningProblem} when $k\to\infty$. Moreover, setting $v_k=\sum_{j=1}^{M_{k}}\theta_{j}^{k}\phi_j$, where $X_k=\{\phi_i\}_{i=1}^{M_k}$, and defining \beq u_{i}^{k}=-\frac{1}{\beta}B^{T}\nabla v_{k}(y_{i}^{k}),\eeq where $y_{i}^{k}$ is the solution of the closed loop problem \eqref{ClosedLoopProblem3} for $v=v_{k}$, $T=T_k$ and $y_0=y_{0}^{i}$, we have that
\begin{equation}
    y_{i}^{k}\weak y^{*}_{i} \mbox{ in }H^{1}_{loc}((0;\infty);\R^{d})\mbox{ and } u_{i}^{k}\weak u^{*}_i \mbox{ in }L_{loc}^{2}((0;\infty);\R^{m}),\mbox{ when }k\to\infty
\end{equation}
where $u_{i}^{*}$ is a solution of the open loop problem \eqref{ControlProblem} and $y_{i}^{*}$ the solution of \eqref{Ode1}.
\label{ConvergenceTheo2}
\end{theo}
\begin{rem}
{\em Assuming that \eqref{LearningProblem} admits a solution $v^{*}\in C^{1,1}(\overline{\Omega})$
and that there exist $v_{n}\in \P_n$ such that
\begin{equation}
     \lim_{n\to \infty}v_{n}=v^{*}\mbox{ in }C^{1,1},\ \mathcal{J}_{\infty}(v^{*}_{n})<\infty\mbox{ and } \lim_{n\to \infty}\mathcal{J}_{\infty}(v_n)=\mathcal{J}_{\infty}(v^{*}),
    \label{DensityHyp}
\end{equation}
it is possible to take $T_k=\infty$ for all $k\in\N$. In this case one can formulate  \eqref{PolyLearningProblem} as infinite horizon problem.
}\end{rem}

\section{Optimality Conditions}
\label{OptimalityConditionsSection}
Throughout this section we consider a basis $X=\{\phi_{i}\}_{i=1}^{M}$ for $M\in \N$, where for each $i\in \{1,\ldots,M\}$ the function $\phi_{i}$ is a monomial given by \eqref{Monomial} for a multi-index $\alpha^i\in \N^{d}$. We recall that we have defined the function $\tilde{\mathcal{J}}_{T}$ such that $ \theta\mapsto\tilde{\mathcal{J}}_{T}(\theta)=\mathcal{J}_{T}(v)$, where $v=\sum_{k=1}^{M}\theta_k\phi_k.$\par
Consider $\theta\in \R^{M}$ and $T>0$ finite. Assume that for each $i\in \{1,\ldots,I\}$, then there exists a unique solution of \eqref{ClosedLoopProblem3}   with $v=\sum_{k=1}^{M}\theta_k\phi_k$, denoted by  $y_{i}\in C^{1}([0,T],\overline{\Omega})$. Further, assume that $y_{i}(t)\in \Omega$ for all $t\in[0,T]$. Then, $\tilde{\mathcal{J}}_{T}$ is differentiable in $\theta$ and  its partial derivatives are given by
\beq\frac{\partial}{\partial \theta_{k}} \tilde{\mathcal{J}}_{T}(\theta)=\frac{1}{I\beta}\int_{0}^{T}\sum_{i=1}^{I} \nabla \phi_{k}^{\top}(y_{i})BB^{\top}\left(\nabla v(y_{i})+p_{i} \right)dt\mbox{ for }k\in \{1,\ldots,M\}, \label{PartialDerivative}\eeq
where $p_{i}$  is the solution of
\beq
-p_{i}'-Df(y_i)^{\top}p+\frac{1}{\beta}\nabla^{2}v(y_{i})BB^{\top}(\nabla v (y_i)+p_{i} )=-\nabla \ell(y_i),\quad p_{i}(T)=0
\label{AdjointEq}
\eeq
for $i\in\{1,\ldots,I\}$, where $Df$ is the Jacobian matrix of $f$.
\par

For the non differentiable term in $P_{\gamma,r}$ we recall that the subgradient of $|\cdot|_1$ is given by
$$ \partial |\cdot|_{1}(\theta)_{k}=\left\{\begin{array}{ll}
    \{1\} & \mbox{if }\theta_k>0, \\
    \{-1\} &\mbox{if } \theta_k<0, \\
    \left[ -1,1\right]  &\mbox{if } \theta_k=0,
\end{array}\right.\mbox{ for }k\in \{1,\ldots,M\}. $$
Hence, if $\theta^{*}\in \R^{M}$ is a solution of \eqref{PolyLearningProblem}  it satisfies the following optimality condition
\begin{equation}
    \nabla \tilde{\mathcal{J}}_{T}(\theta^{*})+\gamma (1-r)\theta^{*}\in- \gamma r\cdot \partial |\cdot|_{1}(\theta^{*}).
    \label{OptCond}
\end{equation}
For each $k\in \{1,\ldots,M\}$, we deduce from  \eqref{OptCond} that
$$\left |\frac{\partial}{\partial \theta_{k}}\tilde{\mathcal{J}}(\theta^{*})  +\gamma(1-r)\theta_{k}^{*}\right|< \gamma r   \implies    \theta_{k}^{*}=0 $$
and
$$\mbox{ if }\theta_{k}^{*}\neq 0,\mbox{ then } \frac{\partial}{\partial \theta_{k}}\tilde{\mathcal{J}}(\theta^{*})  +\gamma(1-r)\theta_{k}^{*}=\left\{\begin{array}{cl}
    \gamma r  & \mbox{ if }\theta_{k}^{*}<0, \\
    -\gamma r & \mbox{ if }\theta_{k}^{*}>0.
\end{array}\right.$$
In the remainder of this section we shall verify the following property which is enjoyed by
any optimal solution $\theta^{*}$:
\beq \mbox{ for each }k\in \{1,\ldots,M\},\ \theta_{k}^{*}=0 \mbox{ if and only if } B^{\top}\nabla\phi_{k}(y)=0 \mbox{ for all } y\in \R^{d}.\label{BOrthogonality}\eeq
We define the subset $\mathcal{O}(X)$  of $X$  by
\beq
\mathcal{O}(X):=\{\phi\in X: \ B^{\top}\nabla\phi(y)=0 \ \forall \ y \in \R^{d}\}
\label{OXSet}
\eeq

It is possible to further characterize  the elements of $\mathcal{O}(X)$.
\begin{lemma}
\label{BOrthogonalityLemma}
Let $\alpha=(\alpha_1,\ldots,\alpha_d)\in \N^{d}$ be a multi-index and consider $\phi_\alpha$ the monomial given by \eqref{Monomial}. Then,
\beq B^{\top}\nabla \phi_{\alpha}(y)=0\mbox{ for all }y\in\R^{d}\mbox{ if and only if } B^{\top}\cdot e_i=0\mbox{ for all }i\in \mathcal{I}(\alpha),\label{OrtVeri}\eeq
where $e_{i}$ is the $i-$th canonical vector of $\R^{d}$ and
$$\mathcal{I}(\alpha)=\{i\in \{1,\ldots,d\}: \alpha_{i}>0\}$$
\end{lemma}
\begin{proof}
Let assume first that
    \beq B^{\top}\nabla \phi_{\alpha}(y)=0 \mbox{ for all }y\in\R^{d}.\label{BOrthogonalityLemmaProof:0}\eeq
    We prove now
\beq B^{\top}  e_i=0\mbox{ for all }i\in \mathcal{I}(\alpha). \label{BOrthogonalityLemmaProof:1}\eeq
If $\mathcal{I}(\alpha)=\{i\}$ for some $i\in \N$, then we have
$ \nabla\phi_{\alpha}(y)=\alpha_iy_i^{\alpha_i-1}e_i$  if $\alpha_i \ge 1$
and therefore \eqref{BOrthogonalityLemmaProof:1} is evident. On the other hand, if $\mathcal{I}(\alpha)$ contains more than one element, we take any $i\in \mathcal{I}(\alpha)$. Then, for $\varepsilon>0$, we define $y^{\varepsilon}\in \R^{d}$ by
$$ y^{\varepsilon}_k=\left\{\begin{array}{ll}
   \dis \varepsilon^{-a} & \mbox{if }k=i, \\
    \ecart\dis\varepsilon & \mbox{if }k\neq i,
\end{array}\right. \mbox{ for all }k\in \{1,\ldots,M\},$$
where in the case $\alpha_i \neq 1$
$$a=\sum_{j\neq i,j=1}^{N}\alpha_j/(\alpha_i-1)>0.$$
Evaluating $\nabla\phi_{\alpha}$ in $y^{\varepsilon}$ we obtain $$ \frac{\partial \phi_{\alpha}}{\partial y_k}(y^{\varepsilon})=\left\{\begin{array}{ll}
   \dis \alpha_i & \mbox{if }k=i, \\
    \ecart\dis\alpha_k\varepsilon^{-a-1} & \mbox{if }k\neq i
\end{array}\right. .$$
Taking $\varepsilon \to \infty$ we obtain
$$ \nabla\phi_{\alpha}(y^{\varepsilon})\to \alpha_i e_{i} $$
and therefore
$$ B^{\top}e_i=\lim_{\varepsilon\to \infty} B^{\top}\nabla\phi_{\alpha}(y^{\varepsilon})=0.$$
If $\alpha_{i}=1$, then we have
$$ \frac{\partial \phi_{\alpha} }{\partial y_{k}}(y)=\left\{\begin{array}{cc}
    \dis\prod_{j\neq i}y^{\alpha_j} & \mbox{ if }i=k, \\ \ecart
    \dis \alpha_k y^{\alpha_k-1}\prod_{j\neq k}y^{\alpha_j} & \mbox{ if }i\neq k,
\end{array}\right.\mbox{ for all }y\in\R^{d}.  $$
We choose $\bar{y}\in \R^{d}$ such that $\bar{y}_{i}=0$ and $\bar{y}_{j}=1$ for all $j\in \{1,\ldots,d\}\setminus\{i\}$.
Evaluating $\nabla\phi_{\alpha}$ in $\bar{y}$ we again obtain
$$ B^{\top}e_i=B^{\top}\nabla\phi_{\alpha}(\bar{y})=0. $$
Since the $i\in \mathcal{I}(\alpha)$ is arbitrary, we have proved \eqref{BOrthogonalityLemmaProof:1}. \par
Now we assume that \eqref{BOrthogonalityLemmaProof:1} holds  and we prove \eqref{BOrthogonalityLemmaProof:0}. For every $j\in \{1,\ldots,d\}\setminus\mathcal{I}(\alpha)$, it is clear that
$$ \frac{\partial\phi_{\alpha}}{\partial y_j}(y)=0\mbox{ for all }y\in \R^{d}.$$
From this, we have that
$$ B^{\top}\nabla\phi_{\alpha}(y)=\sum_{i\in \mathcal{I}(\alpha)}\frac{\partial\phi_{\alpha}}{\partial y_i}(y)B^{\top}e_{i}=0,$$
which concludes the proof.
\end{proof}
\begin{lemma}
\label{BasisReductionLemma}
Consider $\theta\in \R^{M}$ and $T\in (0,\infty)$. Assume that $y_i(t)\in \Omega$ for all $i\in \{1,\ldots,I\}$ and $t\in [0,T]$, where $y_i$ is the solution of the closed loop problem \eqref{ClosedLoopProblem3}  with $v=\sum_{k=1}^{M}\theta_k\phi_k$ and $y_0=y_0^i$. Then,
\beq  \frac{\partial}{\partial \theta_k}\tilde{\mathcal{J}}_{T}(\theta)=0\mbox{ for every }k\in\{1,\ldots,M\}, \mbox{ such that }\phi_k\in \mathcal{O}(X).\label{BasisReductionLemma:1}\eeq
Moreover, if $\theta^{*}\in \R^{M}$ is an optimal solution of \eqref{PolyLearningProblem}, then
\beq \theta^{*}_k =0 \mbox{ for every }k\in\{1,\ldots,M\},\mbox{ such that }\phi_k\in \mathcal{O}(X).\label{BasisReductionLemma:2}\eeq
\end{lemma}
\begin{proof}
For every $\theta\in \R^{M}$, \eqref{BasisReductionLemma:1} is a direct consequence of \eqref{PartialDerivative}. To prove \eqref{BasisReductionLemma:2} we proceed by contradiction. Let $\theta^{*}\in \R^{M}$  be an optimal solution for \eqref{PolyLearningProblem} and assume that there exists   $\bar{k}\in\{1,\ldots,M\}$ such that
\beq B^{\top}\nabla \phi_{\bar{k}}=0\mbox{ in }\overline{\Omega}\mbox{, but } \theta_{\bar{k}}^{*}\neq 0.\label{BOrthogonality2}\eeq
Then, we define $\tilde{\theta}\in\R^{M}$ by
$$\tilde{\theta}_{j}=\left\{\begin{array}{cc}
 \dis\theta_{j}^{*}    & \mbox{ if }j\neq \bar{k} \\ \ecart
\dis 0 & \mbox{ if }j=\bar{k}
\end{array}\right. \mbox{ for all }j\in\{1,\ldots,M\} .$$
By \eqref{BOrthogonality2} we have
$$ \sum_{k=1}^{M}B^{\top}\nabla \phi_{k}(y)\tilde{\theta}_{k}=\sum_{k=1}^{M}B^{\top}\nabla \phi_{k}(y)\theta^{*}_{k}\mbox{ for all } y\mbox{ in }\Omega.$$
Consequently, for each initial condition $y_{0}^{i}$,  the solution of  \eqref{ClosedLoopProblem3} for $\tilde{v}=\sum_{k=1}^{M}\phi_{k}\tilde{\theta}_{k}$ is the same as for $v^{*}=\sum_{k=1}^{M}\phi_{k}\theta^{*}_{k}$. Therefore we get $$ \tilde{\mathcal{J}}_{T}(\tilde{\theta})=\tilde{\mathcal{J}}_{T}(\theta^{*}).$$
Further, it is clear that $P_{\gamma,r}(\theta^{*})>P_{\gamma,r}(\tilde{\theta})$, because $\theta^{*}_{\bar{k}}\neq 0$. Thus
$$  \tilde{\mathcal{J}}_{T}(\tilde{\theta})+P_{\gamma,r}(\tilde{\theta})<\tilde{\mathcal{J}}_{T}(\theta^{*}) +P_{\gamma,r}(\theta^{*}),$$
which is a contradiction.
\end{proof}

\begin{rem}\label{BasisReductionRem}
{\em From Lemma \ref{BasisReductionLemma} we conclude that basis functions $\phi_k \in \mathcal{O}(X)$ do not contribute to the optimal solution of \eqref{PolyLearningProblem}. Therefore they should be dismissed before computing the minimizers \eqref{PolyLearningProblem}. This can be done utilizing Lemma \ref{BOrthogonalityLemma}. In this way  we replace $X$ by $X\setminus\mathcal{O}(X)$.}
\end{rem}

\section{Optimization Algorithm}
\label{OptimizationAlgoSection}
In this section we consider $T\in (0,\infty)$ and $X=\{\phi_{i}\}_{i=1}^{M}$ with $M\in \N$, where for each $i\in \{1,\ldots,M\}$ the function $\phi_{i}$ is a monomial of the form \eqref{Monomial} for a multi-index $\alpha_i\in \N^{d}$. To solve \eqref{PolyLearningProblem} we use a linear proximal point method with an adaption Barzilai-Borwein method for choosing the step length, which proved to be efficient for high dimensional problems (see \citep{AzKK}, \citep{Barzilai}, and \citep{Raydan} for a convergence analysis in the smooth case). In contrast to the smooth setting, we are not aware of a thorough convergence analysis of this particular step size choice in the nonsmooth case. However, from a practical point of view, the method performs reliably for our purposes.  \par

We now describe the algorithm that we use to solve \eqref{PolyLearningProblem}. We denote the $k-$th element of the sequence produced by the algorithm by $\theta^{k}$, the step size by $s^k$, and we define at each iteration
\beq d^{k}:=\nabla \tilde{\mathcal{J}}_{T}(\theta^{k})+\gamma(1-r)\theta^{k}.
\label{gradient}\eeq
We use the proximal point update rule as is described  in section 10.2 in \citep{Beck}, namely we take $\theta^{k+1}$ such that
\begin{equation}
    \theta^{k+1}=\argmin_{\vartheta\in\R^{M}}
    \left\{d^{k}\cdot\left(\vartheta-\theta^k\right) +\frac{1}{2s^k}|\theta^k-\vartheta|_{2}^{2}+\alpha r|\vartheta|_{1}\right\}.
    \label{ProximalPointUpdate}
\end{equation}
Defining
$$ shrink(a,b)=\left\{\begin{array}{cc}
    a-b & \mbox{ if }a-b>0  \\
   a+b& \mbox{ if } a+b<0 \\
    0   & \mbox{ if } |a|\leq |b|.
\end{array}\right. $$
the update rule \eqref{ProximalPointUpdate} can be expressed as
\begin{equation}
\theta_j^{k+1}=shrink\left(\theta^{k}_j -s^{k} d^k_j,s^{k}\gamma r\right).
    \label{ProximalPointUpdate2}
\end{equation}
for each $j=1,\ldots,M.$

If the cardinality of $X$ is large  and $\theta^{k}$ has many non-zero entries, the evaluation of $\mathcal{J}_{T}$ and $\nabla \mathcal{J}_{T}$ can be very expensive. Consequently it is useful  to initialize sparsely and to  monitor the sparsity level during the iterations of the algorithm.
The $\ell^1$ term will enhance sparsity in the limit. During the iterations we only update one coordinate $j^{k}$ chosen by a greedy rule proposed in \citep{WuLa} (see also \citep{ShiTuXu}), in order to keep $\theta^{k}$ as sparse as possible. Namely, instead of updating  all the coordinates of $\theta^{k+1}$ by rule \eqref{ProximalPointUpdate2}, we  determine the coordinates to be updated by \eqref{ProximalPointUpdate2} by means of
\beq  j^{k+1}\in \argmax_{j\in \{1,\ldots,M\}}\min_{z\in \partial|\cdot|(\theta_j)}\left|d_{j}^{k}+z\right|.\label{GreedyRule}\eeq
\par
Concerning initialization of $\theta$ it is not always possible to do this by 0 since the solution of \eqref{ClosedLoopProblem3} with $v=0$ could have a large norm causing numerical difficulties due to the  evaluation of the polynomials or it may not exist for all $t\in [0,T]$. For this reason an initial guess for $\theta$ has to be chosen that at least ensures the boundedness of the solutions of the closed loop problems \eqref{ClosedLoopProblem3}. This depends on the nature of $f$ and the choice of $T$ in \eqref{PolyLearningProblem}. We shall return to this point in the course of discussing the numerical examples.

To choose the step size $s^k$ we use the backtracking line search described in section 10.3.3 in \citep{Beck}, starting from an initial guess $s_{0}^{j}$.
 That is, for $\kappa\in (0,1)$ and $\beta\in (0,1)$, we take $s^{k}=s_0^{k}\beta^{i}$ such that $i$ is the smallest natural number which satisfies
\beq \tilde{\mathcal{J}}_{T}(\theta^{+}) \leq \tilde{\mathcal{J}}_{T}(\theta^{k}) - \frac{\kappa}{s_0^{k}\beta^{i}}|\theta^{k}-\theta^{+}|^2, \label{Backtracking}\eeq
where either all or only the coordinate determined by \eqref{GreedyRule} are updated by
\eqref{ProximalPointUpdate2}. We use the Barzilai-Borwein step size as initial guess, namely we take $s_0^k$ as
\beq  s_{0}^{k}=\left\{\begin{array}{ll}
     \dis \big[(\theta_{k}-\theta_{k-1})\cdot(d_{k}-d_{k-1})\big]/|d_{k}-d_{k-1}|^{2}   &\mbox{ if } k\mbox{ is }odd,\\
      \ecart\dis |\theta_{k}-\theta_{k-1}|^{2}/\big[(\theta_{k}-\theta_{k-1})\cdot(d_{k}-d_{k-1})\big] & \mbox{ if }k\mbox{ is }even.
\end{array}\right. \label{BB}\eeq
We summarize the algorithm as follows:
\begin{algorithm}[H]
\caption{Sparse polynomial learning algorithm.}
\label{Alg1}
\begin{algorithmic}[1]
\Require An initial guess $\theta^{0}\in\R^{M}$,$\gamma>0$, $\kappa>0$, $\beta\in (0,1)$, $T>0$, $r\in [0,1]$, $s_0\in (0,\infty)$.
\Ensure An approximated stationary point $\theta^{*}$ of \eqref{PolyLearningProblem}.
\State $k=1 $
\State For $\theta^0$, set $d^0$ using \eqref{gradient}
and $J_0=\tilde{\mathcal{J}}_{T}(\theta^0)+P_{\gamma,r}(\theta^0).$
\State Use \eqref{Backtracking} to obtain $s_{0}$.
\State Use \eqref{ProximalPointUpdate} or \eqref{ProximalPointUpdate2} with $j=j^{1}$ given by \eqref{GreedyRule} to get $\theta^{1}$.
\State Obtain $d^{1}$ by using \eqref{gradient} and set $J_1=\tilde{\mathcal{J}}_{T}(\theta^1)+P_{\gamma,r}(\theta^1).$
\While{ $|d^k|>gtol$ and $|J_k-J_{k-1}|>tol $ }

\State Obtain $s_{0}^{k}$ by using \eqref{BB} and  choose $s^{k}$ using \eqref{Backtracking}.
\State Use \eqref{ProximalPointUpdate} or \eqref{ProximalPointUpdate2} with $j=j^{k+1}$ given by \eqref{GreedyRule} to get $\theta^{k+1}$.
\State Obtain $d^{k+1}$ by using \eqref{gradient} and set $J_{k+1}=\tilde{\mathcal{J}}_{T}(\theta^{k+1})+P_{\gamma,r}(\theta^{k+1}).$
\State Set $k=k+1$.
\EndWhile
\Return $\theta^{*}:=\theta_{k}$.

\end{algorithmic}

\end{algorithm}

%\begin{rem}
%{\em For high dimensional problems, the evaluation of $\tilde{\mathcal{J}}_{T}$ and $\nabla\tilde{\mathcal{J}}_{T}$ could be very expensive. Therefore, it is crucial to keep $\theta_k$ as sparse as possible. In this case, it is preferable to use a sparse initial guess, for example $\theta^0=0$ and update only one coordinate of $\theta^k$ by \eqref{ProximalPointUpdate2} with $j=j^{k+1}$. }
%\end{rem}

\section{Polynomial Basis Evaluation}
\label{EvaluationSection}
Concerning the implementation, we address the problem of an efficient evaluation of $\tilde{\mathcal{J}}(\theta)$ and $\nabla\tilde{\mathcal{J}}(\theta)$. Indeed, to evaluate $\tilde{\mathcal{J}}$ and $\nabla\tilde{\mathcal{J}}$, we need to solve \eqref{ClosedLoopProblem3} and \eqref{AdjointEq}. We solve these systems numerically, which involves multiple evaluations of the elements of the basis $\B_n$ or $\S_n$ and their derivatives. Therefore it is essential to do this efficiently. For simplicity, we only describe how to evaluate the elements of $\B_{n}$, but the case of $\S_{n}$ is analogous. Our approach for polynomial evaluation is related to \citep{Carnicer} and \citep{Lodha}.\par
Before describing how we evaluate the elements of $\B_n$, we need to recall some concepts from graph theory. We only give some basic definitions following \citep{Rosen},  and refer to  \citep{KoVy} for further description.\par

A directed graph $G=(V,E)$ is a pair, where $V$ is the set of nodes or vertices of the graph and $E\subset V\times V$ is the set of edges of $G$.\par  For a graph $G=(V,E)$ a directed path that connects $a\in V$ and $b\in V$ is a sequence of vertices  $\{v_{i}\}_{i=1}^{k}\subset V$ such that $(v_{i},v_{i+1})\in E$  for all $i\in \{1,\ldots,k-1\}$, $a=v_{1}$ and $b=v_k$, furthermore we say that a directed path is a directed circuit or cycle if $a=b$. Similarly, an undirected path that connects $a$ and $b$ is a sequence of vertices  $\{v_{i}\}_{i=1}^{k}\subset V$ such that $(v_{i},v_{i+1})\in E$ or $(v_{i+1},v_{i})\in E$  for all $i\in\{1,\ldots,k-1\}$, $a=v_{1}$ and $b=v_k$, furthermore we say that an undirected path is a circuit or cycle if $a=b$.\par
A directed rooted tree is a graph $G=(V,E)$ such that there is no undirected circuit in $G$ and it has a node $v_r\in V$ which is connected to every $v\in V\setminus\{v_r\}$.\par
For two graphs $G=(V,E)$ and $G'=(V',E')$, we say that $G$ is a subgraph of $G$ if $V'\subset V$ and $E'\subset E$.\par
For a graph $G=(V,E)$, a minimum spanning rooted tree is a subgraph $T=(V,E')$ of $G$ such that $T$ is a rooted tree. \par
We also need to recall a fundamental algorithm to traverse a graph, which is called the breadth-first search (BFS),~\citep{KoVy}.  For $G$ a directed graph and $r$ a node in $G$ connected to every other node in $G$, the BFS algorithm returns a minimum spanning rooted tree with $r$ as its root.
\begin{algorithm}[H]
\begin{algorithmic}[1]
\Require A graph $G=(V,E)$ and $v_r\in V$.
\Ensure A subgraph $G'=(V,E')$
\State Set $E'=\emptyset$
\State For every $v\in V$ set $color(v)=0$.
\State Choose~$v_r \in V$, set~$I=1$, and~$q=\{(1,v_r)\}$.

\While{ $q\neq \emptyset$.}
%\State Denote the first element of $q$ by $v$.
\State Set $v$ to be such $(1,v)\in q$.
\For{$\tilde{v}\in V$ such that $(v,\tilde{v})\in E$}
\If{ $color(\tilde{v})=0$}
\State Set $color(\tilde{v}):= 1$.
\State Set~$I:=I +1$,~$q:=q\cup \{(I,\tilde{v})\}$.
\State Set $E':=E'\cup \{(v,\tilde{v})\}$
\EndIf
\EndFor
\State Set $q:=q\setminus \{(1,v)\}$,~$I:=I-1$ and ~$q:=\{(i-1,u): \forall (i,u)\in q\}$.

\EndWhile
\Return $V'$
\end{algorithmic}
\label{BFS}
\caption{Breadth-first search (BFS)}
\end{algorithm}
We are now prepared to describe the evaluation of all the elements of $X$ in a given point $y\in \R^{d}$. We recall that for simplicity we only consider the case $X=\B_{n}$, later we explain how to do it in other cases.

Let us consider the directed  graph $G=(\Lambda_n,E_n)$ where $\Lambda_n$ is given \eqref{AlphaIndexes} and $ E_{n}\subset \Lambda_{n}\times\Lambda_n$ is defined by \beq \forall \ \tilde{\alpha},\alpha\in\Lambda_n:\   (\tilde{\alpha},\alpha)\in E_{n} \mbox{ if and only if } \alpha=\tilde{\alpha}+e_j\mbox{ for an unique }j\in \{1,\ldots,d\},\label{EvaluationTreeEdgeDef}\eeq
where $e_j$ is the $j$-th canonical vector of $\R^{d}$. Let $T$ be a  minimum spanning rooted tree of $G$, where $\alpha^{0}=(0,\ldots,0)\in \N^{d}$ is the root of $T$. Then, for every $\alpha\in \Lambda_{n}\setminus\{\alpha^{0}\}$ there exists a unique  $\tilde{\alpha} \in \Lambda_{n}$ such that $(\tilde{\alpha},\alpha)$ is a vertex of $T$, which in turn implies that there exists $j\in \{1,\ldots,d\}$ such $\alpha=\tilde{\alpha}+e_j$. Therefore we have
\beq \phi_{\alpha}(y)=y_j^{\alpha_j}\prod_{i=1,i\neq j}^{d}y_i^{\alpha_i}=y_j\cdot y_j^{\alpha_j-1}\prod_{i=1,i\neq j}^{d}y_i^{\alpha_i}=y_j\phi_{\tilde{\alpha}}(y),\ \forall \ y  \in \R^{d}. \label{Evaluationrule}\eeq
For a given $y\in\R^{d}$ we evaluate all the elements of $\B_{n}$ by performing a BFS in $T$ starting from the root and using \eqref{Evaluationrule}. More precisely, we denote by $c(\alpha)$ the value of $\phi_{\alpha}(y)$. It is clear that $c(\alpha^{0})=1$. Now, let $\tilde{\alpha}^{i}$ be the node visited in the $i-$th iteration of BFS. Then for each element in $\alpha\in\Lambda_n$ such that $(\tilde{\alpha}^{i},\alpha)\in T$ we obtain $c(\alpha)$ using \eqref{Evaluationrule}, i.e.
\beq c(\alpha)=y_jc(\tilde{\alpha}^{i}),\label{Evaluationrule2}\eeq
where $j$ is the unique index such that $\alpha=\tilde{\alpha}^{i}+e_j$.\par

 Similarly to \eqref{Evaluationrule}, the partial derivatives of $\phi_{\alpha}$ satisfy
\beq \frac{\partial\phi_{\alpha}}{\partial y_{i}}(y)=\left\{\begin{array}{ll}
    \dis \alpha_{i}\phi_{\alpha-e_{i}}(y) & \mbox{ if }\alpha_{i}>0,  \\
    \ecart\dis  0 & \mbox{ if } \alpha_{i}=0
\end{array}\right.
\label{EvaluationruleD} \mbox{ for all }i\in \{1,\ldots,d\} \eeq
and
\beq
\frac{\partial^{2}\phi_{\alpha}}{\partial y_{i}\partial y_{j}}(y)=\left\{\begin{array}{ll}
    \dis\alpha_{i}\alpha_{j}\phi_{\alpha-e_{i}-e_{j}}(y) & \mbox{ if }i\neq j,\ i\geq 1\mbox{ and }j\geq 1 , \\
     \ecart\dis \alpha_{i}(\alpha_{i}-1)\phi_{\alpha-2e_{i}}(y) & \mbox{ if }i=j \mbox { and }\geq \alpha_{i}=2,\\
     \ecart 0&\mbox{ if }i=j\mbox{ and } \alpha_{i}=1,\\
     \ecart 0& \mbox{ if }i\neq j\mbox{ and }\alpha_{i}=0\mbox{ or }\alpha_{j}=0,
     \end{array}\right.
     \label{EvaluationruleDD}
\eeq
for all $i,j\in \{1,\ldots,m\}$, where $e_{i}\in \R^{d}$ is the $i-th$ canonical vector in $\R^{d}$. Therefore, by the definition of $c$, we have
\beq \frac{\partial\phi_{\alpha}}{\partial y_{i}}(y)=\left\{\begin{array}{ll}
    \dis \alpha_{i}c(\alpha-e_{i}) & \mbox{ if }\alpha_{i}>0,  \\
    \ecart\dis  0 & \mbox{ if } \alpha_{i}=0
\end{array}\right.
\label{EvaluationruleD2}\forall i \in \{1,\ldots,d\} \eeq
and
\beq
\frac{\partial^{2}\phi_{\alpha}}{\partial y_{i}\partial y_{j}}(y)=\left\{\begin{array}{ll}
    \dis\alpha_{i}\alpha_{j}c(\alpha-e_{i}-e_{j}) & \mbox{ if }i\neq j,\ i\geq 1\mbox{ and }j\geq 1 , \\
     \ecart\dis \alpha_{i}(\alpha_{i}-1)c(\alpha-2e_{i})& \mbox{ if }i=j \mbox { and }\geq \alpha_{i}=2,\\
     \ecart 0&\mbox{ if }i=j\mbox{ and } \alpha_{i}=1,\\
     \ecart 0& \mbox{ if }i\neq j\mbox{ and }\alpha_{i}=0\mbox{ or }\alpha_{j}=0,
     \end{array}\right.
     \label{EvaluationruleDD2}
\eeq
for all $i,j\in \{1,\ldots,d\}$.\par
Now we address the evaluation of the elements of $\S_n$ and their derivatives. We consider a graph $\bar{G}=(\Gamma_n,\bar{E}_n)$, where $\Gamma_n$ is given by \eqref{HyperCrossMultiindex} and $\bar{E}_n\subset\Gamma_n\times\Gamma_n$ is defined by
\beq \forall \ \tilde{\alpha},\alpha\in\Gamma_n:\   (\tilde{\alpha},\alpha)\in \bar{E}_{n} \mbox{ if and only if } \alpha=\tilde{\alpha}+e_j\mbox{ for an unique }j\in \{1,\ldots,d\}.\eeq
By the definition of  $\Gamma_n$, it is clear that for all $\alpha\in \Lambda_n\setminus{\alpha^{0}} $, there exists at least one $\tilde{\alpha\in\Lambda}$ which satisfies $\alpha=\tilde{\alpha}+e_j$ for some $j\in\{1,\ldots,m\}$, therefore $\alpha^{0}$ is connected to every $\alpha\in\Lambda_n$. Then, the evaluation of the elements of $\S_n$ is analogous to the evaluation of $\B_n$.\par
Finally, in virtue of \Cref{BasisReductionRem}, we know that not all the elements of either $X=\S_{n}$ or $X=\B_n$ are contributing to the optimal solution. In this case we should consider a reduced basis given by $\mathcal{O}(X)$ in \eqref{OXSet}. Nevertheless, we can not use our approach directly, because it is not possible to ensure that we can construct an spanning tree with only the multi-indexes that correspond to $X\setminus\mathcal{O}(X)$. More generally, for a given $\theta\in \R^{M}$ we only need to evaluate the intersection between $X\setminus\mathcal{O}(X)$ and $\{\phi_{i}\in X:\ \theta_i\neq 0\}$, i.e. the intersection of the support of $\theta$ and $X\setminus\mathcal{O}(X)$. To address this problem, consider a subset $\tilde{X}$  of $X$ and a spanning tree $T$ rooted at $\alpha^{0}$ for either $\Lambda_n$ or $\Gamma_n$, as appropriate. Then we extract a sub-tree $\tilde T$ from $T$ by traversing $T$ starting from each element of $\tilde{X}$.
\section{Generalization}
\label{GeneralizationSection}
When learning approximation schemes on a finite training set, it is of special interest whether the design objective can also be accomplished for configurations which are not contained in the training set. In our case this amounts to achieving stable trajectories for initial data outside of the training set.

To demonstrate a scenario of what can be expected we focus on the case when we train with only one initial condition $\tilde{y}_{0}\in\Omega$.
 We consider $X=\{\phi_{i}\}_{i=1}^{M}\subset C^{1,1}(\overline{\Omega})$ for $M\in \N$, $\theta^{*}\in\R^{M}$, and $v_{*}=\sum_{i=1}^{M}\theta_i^{*}\phi_i$. For $y_0\in\R^{d}$, we denote the solution of \eqref{ClosedLoopProblem3} with $v=v_{*}$ by $y(\cdot,y_0)$.\par
We assume that there exists a neighbourhood around $\tilde{y}_0$ and $0$ where the system \eqref{ClosedLoopProblem3}
is exponentially stable, namely
\beq
\exists \rho>0,\ \kappa>0,\ K\geq 0 \mbox{ such that }|y(t,y_{0})|\leq K e^{-\kappa t}|y_0|\mbox{ for all }y_0\in B(\tilde{y}_0,\rho)\cup B(0,\rho).
\label{ExponentialNei}\eeq
 We now provide a sufficient condition guaranteeing that there exists an open neighbourhood around the trajectory $\mathcal{T}=\{y(\cdot,\tilde{y}_0)(t):\ t>0\}\subset\R^{d}$ such that for all initial conditions in this neighbourhood the solution is exponentially stable.\par
For an arbitrary $\tau>0$ and $\delta x>0$ we consider the linearised system
\beq x'=A(t)x ,\quad x(0)=\delta x, \label{Linearized system} \eeq
where $A(t)=Df({y}(t,\tilde y_0))-\frac{1}{\beta}BB^{\top}\nabla^{2}v_{*}({y}(t,\tilde y_0))$, and define the associated solution mapping $ S(\tau)\delta x=x(\tau,\delta x).$

\begin{prop}
\label{GeneralizationPropKK}
Assume that \eqref{ExponentialNei} holds for $\tilde y_0$.  Then, there exits $\tilde{\rho}>0$  such that
\beq  |y(t,y_0)|\leq K e^{-\kappa t}|y_0|, \mbox{ for all }y_0 \in N(\overline{\mathcal{T}}),\label{GeneralizationPropKK:1}\eeq
where $N(\overline{\mathcal{T}})=\{ y_0\in\R^{d}:\ dist(y_0,\overline{\mathcal{T}})<\tilde{\rho}\}$.
\end{prop}
\begin{proof}
For arbitrary $\tau>0$ define the set
$$ B(\tau)=\{y(\tau,y_0):\ y_0\in B(\tilde{y}_0,\rho)\}. $$
Note that $y(\tau,\tilde{y}_0)\in B(\tau)$ and that by the Bellman's principle
$$ |y(t,y_0)|\leq K e^{-kt}|y_0|, \ t>0,\mbox{ for all }y_0\in B(\tau).$$
We need to argue that $\inf_{\tau >0} diam (B(\tau)) >0$.
For this purpose we apply the implicit function theorem to the mapping
$$  G:B(\tilde{y}_0,\rho)\times\R^{d}\subset\R^{d}\times\R^{d}\to \R^{d} $$
$$ G(y_0,z)=y(\tau,y_0)-z, $$
to argue that $y(\tau,\tilde{y}_0)\in int(B(\tau)).$ Indeed, $G(\tilde{y}_0,y(\tau,\tilde{y}_0))=0$ and $DG_{y_0}$ is characterized by
$$ DG_{y_0}\delta x=S(\tau)\delta x. $$
By Liouville's theorem $S(\tau)$ is an isomorphism. Hence for each $\tau>0$  there exits $\rho_{\tau}>0$ such that $B(y(\tau,\tilde{y}_0),\rho_{\tau})\subset B_{\tau}$. By construction,
$$ |y(t,y_0)|\leq Ke^{-\kappa t}|y_0|,\ \forall\ t>0,\forall \ y_0\in B(y(\tau,\tilde{y}_0)\rho_{\tau}).$$
We consider the covering $$\bigcup_{\tau>0}B(y(\tau,\tilde{y_0},\rho_{\tau}))\cup B(\tilde{y_0},\rho))\cup B(0,\rho)\supset \bar{\mathcal{T}}.$$
Since $\lim_{\tau \to \infty} y(\tau,\tilde y_0)=0$,  the set $\bar{\mathcal{T}}$ is compact. Therefore there exits a finite subcover, and consequently some $\tilde{\rho}>0$ such  that \eqref{GeneralizationPropKK:1} holds.
\end{proof}

\begin{rem}
{\em In the case that $\{\tilde y_0^i\}_{i=1}^I$ initial conditions are used for the learning step,  the construction of Theorem \ref{GeneralizationPropKK} can be repeated for each one of them leading to $I$ tubes containing the trajectories $\{y(\cdot,\tilde y_0^i)\}_{i=1}^I$. For initial conditions in these tubes we have guaranteed exponential stabilization. Moreover, since these tubes all intersect at the origin it can be expected that, as $I$ increases, the neighorhood of the origin for which stabilization is guaranteed increases as well. }
\end{rem}

\section{Numerical Experiments}
\label{NumericalExperimentsSection}
We implement \Cref{Alg1} to solve problem \eqref{PolyLearningProblem} for 4 different problems. These problems are the stabilization of an LC-circuit, stabilization of a modified Van der Pol oscillator, stabilization of the Allen-Cahn equation, and optimal consensus for the Cucker-Smale model. For every experiment, we shall specify the computational time horizon, and the sets of initial conditions for training and testing.  For  all experiments the ordinary differential equations are solved by the Crank-Nicolson algorithm.  The arguments of the monomials  are normalized by $l$,  i.e. we redefine $\phi_{\alpha}(y)$ by
$\phi_{\alpha}(y)=\prod_{i=1}^{d}\left(\frac{y}{l}\right)^{\alpha_i}$.
\par We measure the performance of our approach by comparing the control $u^*$  obtained by solving  the open loop problem for every initial condition in the test set with
%the one obtained through \eqref{HJBFeedbackLaw} by replacing $V$ by the solution of the learning problem \eqref{PolyLearningProblem}. That is, for every initial condition $y_0$ in the test set, we consider $u^{*}$ the solution of problem \eqref{ControlProblem} and
$\hat{u}=-\frac{1}{\beta}B^{\top}\nabla \hat{v}(y)$, where $\hat{v}$ is the solution of \eqref{PolyLearningProblem} and $y$ is the corresponding solution of \eqref{ClosedLoopProblem3}. We then compute the mean normalized squared error in $L^{2}((0,T);\R^{m})$ for the controls  by
$$
SSE_u(\{\hat{u}_{i}\}_{i=1}^{N},\{u_i^{*}\}_{i=1}^{N})=
\sum_{i=1}^{N}\int_{0}^{T}|\hat{u_i}-u_i^{*}|^{2}dt\Big/\sum_{i=1}^{N}\int_{0}^{T}|u_i^{*}|^{2}dt,
$$
and analogously $SSE_y(\{\hat{y}_{i}\}_{i=1}^{N},\{y_i^{*}\}_{i=1}^{N})$ for the states.
We also compare the optimal value of the open loop problem with the  objective function of \eqref{ControlProblem} evaluated in $\hat{u}$ by computing the mean normalized squared error
$$ SSE_J(\{\hat{u}_{i}\}_{i=1}^{N},\{u_i^{*}\}_{i=1}^{N})=\sum_{i=1}^{N}
|J(u_i^{*},y_0^{i})-J(\hat{u}_i,y_0^i)|^2\Big/\sum_{i=1}^{N}J(u_i^{*},y_0^{i})^2.$$
\par
%We recall that for a linear-quadratic problem of the form
%\beq
%\begin{array}{c}
%    \dis \min_{u\in L^{2}((0,\infty),\R)}\frac{1}{2} \int_0^{\infty}|y|^{2}dt+\frac{\beta}{2} \int_0^{\infty}|u|^2 dt \\ \ecart
%    \dis s.t.\ y'= Ay+Bu,\ y(0)=y_0,
%\end{array}
%\label{LQP}
%\eeq
% the value function is given by $V(y)=\frac{1}{2}y^{\top}Ky$, where $K\in \R^{d\times d}$ is symmetric and positive definite matrix which solves the Riccati equation:
%\beq
%A^{\top}K+KA-\frac{1}{\beta}KK=I.\label{RiccatiEq}
%\eeq
%Therefore, to obtain the value function we solve \eqref{RiccatiEq} and then we use \eqref{HJBFeedbackLaw} to get the optimal control.
In order to compute an optimal control for the non-linear problems, we solve the open loop problem by  a gradient descent algorithm with a backtracking line-search. For the linear-quadratic problem \eqref{LQP} we use the algebraic Riccati equation to obtain the optimal feedback controls.  \par

\subsection{LC-circuit} We consider the linear-quadratic problem
\beq
\begin{array}{c}
    \dis \min_{u\in L^{2}((0,T),\R)}\frac{1}{2} \int_0^{T}|y|^{2}dt+\frac{\beta}{2} \int_0^{T}|u|^2 dt \\ \ecart
    \dis s.t.\ y'= Ay+Bu,\ y(0)=y_0,
\end{array}
\label{LQP}
\eeq
with
\beq A=\left(\begin{array}{ccc}
     0 & 1 & -1\\
    -1 & 0 & 0 \\
     1 & 0 & 1
\end{array}\right) \mbox{ and } B=\left(\begin{array}{c}
     0  \\
     1 \\
     0
\end{array}\right).\eeq
\par
 Here we pay special attention to the convergence of the  learning problem when the cardinality of the training set increases. We set $T=10$, $l=10$, $\gamma=10^{-30}$, and $r=0.1$,  and randomly choose  two sets of $10$ and $100$ initial conditions from $\overline{\Omega}$ and call these sets $\mathcal{Y}_{train}$ and $\mathcal{Y}_{test}$.
 Here we initialize with $v_0=0$.
 From $\mathcal{Y}_{train}$ we take $\{\mathcal{Y}_i\}_{i=1}^{20}$ as a sequence of increasing subsets of $\mathcal{Y}_{train}$, such that for each $i\in\{1,\ldots,20\}$ the cardinality of $\mathcal{Y}_i$ is $i$. For each $i\in\{1,\ldots,20\}$ we solve \eqref{PolyLearningProblem} with $X=\B_{2}\setminus\left(\mathcal{B}_1\cup\mathcal{O}(\B_{2})\right)$, where $\mathcal{O}(\cdot)$ is given by  \eqref{OXSet}. Subsequently for every $y_0\in\mathcal{Y}_{test}$  the control obtained through the solution of \eqref{PolyLearningProblem} is compared  with the optimal one.
\par
In \Cref{LCcircuit:ErrorTableTest} we present mean normalized errors
 for the objective function, the controls, and the states, for  increasing training sizes. In all the cases where the cardinality of the training set is bigger than one, the mean percentage error is smaller than 2. In \Cref{LCcircuit:MeanPerErr2} we show the error when the cardinality of the training set is bigger than one.

 % Now, it is observed that all the error measures are decreasing with the cardinality of the training set. Furthermore, when training with more than one initial conditions the error abruptly decrease below 1\%.
\begin{table}[h!]
\centering
 \begin{tabular}{ |c||c|c|c|  }\hline
Training size & control error (\%)& state error (\%)& objective error (\%)\\ \hline
1 & 17.56778 & 15.40232 & 2.78807\\
2 & 1.40877 & 0.38278 & 0.00123\\
%3 & 1.01687 & 0.21817 & 0.00019\\
%4 & 1.09240 & 0.22996 & 0.00023\\
5 & 1.33517 & 0.27748 & 0.00045\\
%6 & 0.81598 & 0.15816 & 0.00004\\
%7 & 0.76440 & 0.15415 & 0.00004\\
%8 & 0.72320 & 0.13479 & 0.00002\\
%9 & 0.48073 & 0.07802 & 0.00001\\
10 & 0.45580 & 0.07792 & 0.00001\\
%11 & 0.49746 & 0.08886 & 0.00001\\
%12 & 0.44647 & 0.07147 & 0.00002\\
%13 & 0.45224 & 0.07257 & 0.00002\\
%14 & 0.79881 & 0.14380 & 0.00002\\
%15 & 0.45511 & 0.06965 & 0.00002\\
%16 & 0.49832 & 0.08077 & 0.00001\\
%17 & 0.53694 & 0.08853 & 0.00001\\
%18 & 0.54173 & 0.08811 & 0.00001\\
%19 & 0.51051 & 0.08110 & 0.00001\\
%20 & 0.42850 & 0.06598 & 0.00002\\
\hline\hline\end{tabular}
\vspace{3pt}
\caption{$SSE_u, SSE_y, SSE_J$ in percent for LC-circuit example.}
\label{LCcircuit:ErrorTableTest}
\end{table}

\begin{figure}[h]
    \includegraphics[width=0.9\textwidth]{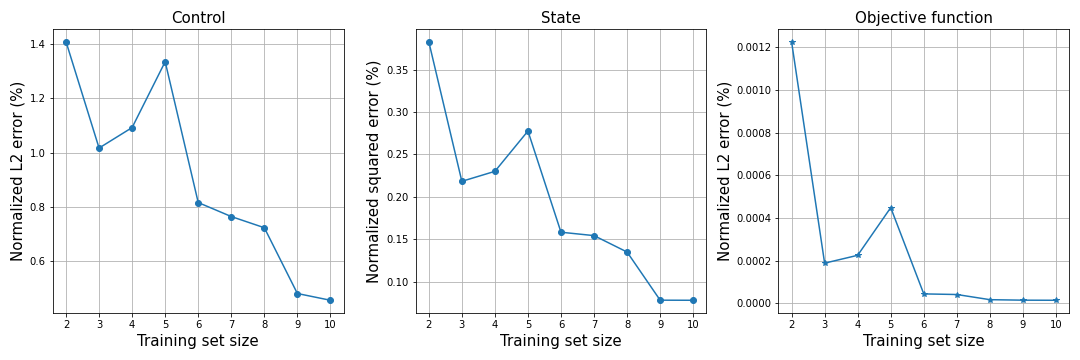}
\caption{$SSE_u, SSE_y, SSE_J$ for LC-circuit example.}
\label{LCcircuit:MeanPerErr2}
\end{figure}

To further illustrate  the performance of our approach, in \Cref{LCcircuit:Scatter}, we provide the scatter plot between the true value of the open loop problem and the objective function evaluated in the learned control for every point in the test set. This carried out  for $\mathcal{Y}_1$, $\mathcal{Y}_2$ , and $\mathcal{Y}_{20}$.  For $\mathcal{Y}_2$ , and $\mathcal{Y}_{20}$ the regression line of the scatter points in the test set are close to the identity line, which is already suggested by the results in Table \ref{LCcircuit:ErrorTableTest}.

\begin{figure}
\includegraphics[width=0.95\textwidth]{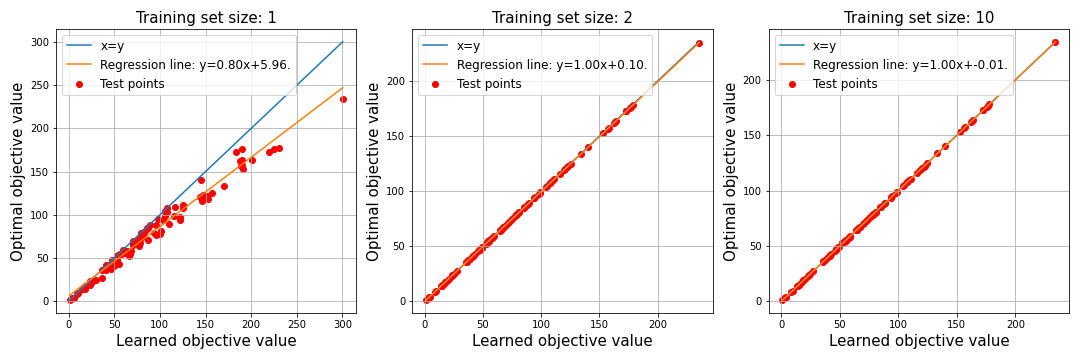}
\caption{Validation scatter for LC-circuit example.}
\label{LCcircuit:Scatter}
\end{figure}

\subsection{Modified Van der Pol Oscillator}
We investigate the stabilization problem
\beq
\begin{array}{c}
\dis\min_{u\in L^{2}((0,T);\R)} \frac{1}{2}\int_{0}^{T}|y|^{2}dt+\frac{\beta}{2}\int_{0}^{T}|u|^{2}dt\\
\ecart \dis s.t. \quad y''=\nu (1-y^{2})y'-y+\mu y^3+u, \ (y(0),y'(0))=(y_0,v_0).\\
\end{array}
\eeq

%Writing this problem as \eqref{ControlProblem}, we have
%$$ f(y_1,y_2)=\left(\begin{array}{c}
%     y_2  \\
%     \nu (1-y_1^{2})y_2-y_1+\mu y_1^3
%\end{array}\right),\quad B=\left(\begin{array}{c}
%     0  \\
%     1
%\end{array}\right),\ \ell(y_1,y_2)=\frac{1}{2}y_1^{2} .$$
In the previous linear-quadratic example we  addressed the convergence when the cardinality of the training set increases. Given the structure of the problem we only used polynomials of degree 2. In the example
 we investigate the effect of the degree of the polynomials.
\par
  The parameters are set to be $T=3$, $\beta=10^{-3}$, $\nu=\frac{3}{2}$, $l=10$, $\mu=\frac{4}{5}$,  and $X_n=\B_n\setminus(\mathcal{B}_1\cup\mathcal{O}(\B_n))$ for $n\in \{4,5,6,7,8\}$. We sample uniformly at random in $\Omega$ a set of $5$ initial conditions as training set and a set of 100 initial conditions as test set. As in the previous example, we take an increasing sequence of subsets of the training set with cardinalities from 1 to 5.
 It is important to mention that if we choose $v=0$ as initial guess, the norm of the solutions of the closed loop problem \eqref{ClosedLoopProblem3} may increase exponentially with time. Therefore, in order to ensure the boundedness of the state of the closed loop problem we proceed as follows, for $X_{4}$ we choose $v_0(y_1,y_2)=\mu\beta y_1^{3}y_2+\frac{\beta\nu}{2}y_2^{2}$ as initial guess and for $n\geq 4$ we use the solution of the previous degree as initial guess. Note that this  requires the polynomial degree to be at least 4.
\par
In \Cref{VanDerPol:TrainErrorplot} and \Cref{VanDerPol:TestErrorplot}  the training  and test errors are depicted. The errors when training with one initial condition are not shown in these figures, because in this case the learned feedback fails to stabilize many of test initial conditions. It is observed that the training errors $SS_J$ and $SS_u$ are decreasing with the degree  when training with more than one initial condition. On the other hand, the training $SS_y$ is not deceasing, but it stays small. We recall that \Cref{ConvergenceTheo2}  assures convergence for the objective function of the training initial conditions, which is consistent with the third subplot of \Cref{VanDerPol:TrainErrorplot}.

In \Cref{VanDerPol:TestErrorplot} the errors  are small for all the training subsets.  The errors are decreasing with the degree once  the cardinality of the initial conditions is larger than four.

In the first graph of \Cref{VanDerPol:SupportCard} the cardinality of the support $\{i:\theta_i\neq 0 \}$ of the  coefficients of the optimal feedback mapping is presented. The second one corresponds to the same cardinality expressed as a percentage of the cardinality of $X_n$. It is observed that the cardinality of the support increases with the degree, but it decreases as percentage of the cardinality of $X_n$. Moreover, for each degree, the cardinality decreases as the  training size increases. The cardinality of $X_n$ equals $9,14,20,27,35$ for $n= 4,\dots,8.$
\par

We turn our attention to the phase planes in \Cref{VanDerPol:Phaseplane}. The first phase plane is composed by the optimal trajectories for the open loop problem, the other ones are composed of the trajectories of the solutions to the closed loop problems for 1, 2 and 5 initial conditions. In the first phase plane we observe that there is a one dimensional manifold with two branches converging to 0 and all the trajectories converges to this manifold. A similar behaviour is seen in the third and forth phase planes. On the other hand, when training with one initial condition, it is seen that the learned feedback law is not capable of stabilizing the test initial conditions near the right-hand branch of the manifold.
\par
We conclude that when training with two or more initial conditions we are able to stabilize all the initial conditions in the test set.

\begin{figure}
\includegraphics[width=0.99\textwidth]{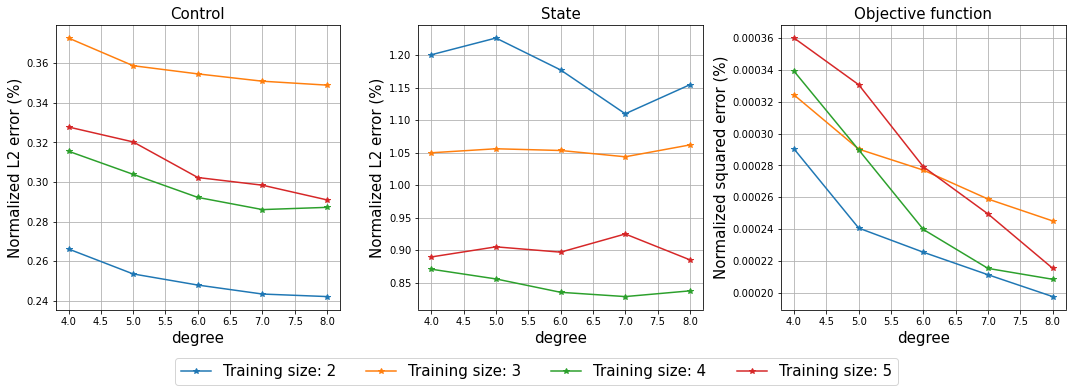}
\caption{Training errors $SSE_u$, $SSE_J$ and $SSE_y$ for modified Van der Pol oscillator example.}
\label{VanDerPol:TrainErrorplot}
\end{figure}

\begin{figure}
\includegraphics[width=0.99\textwidth]{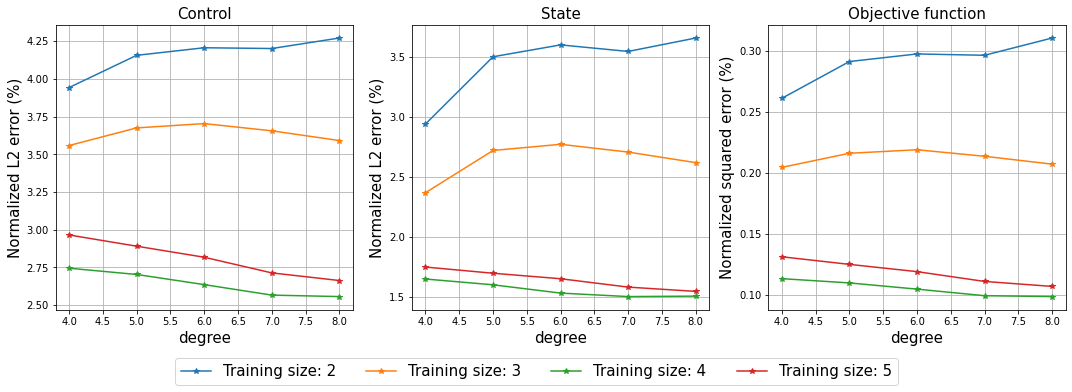}
\caption{Test errors $SSE_u$, $SSE_J$ and $SSE_y$ for modified Van der Pol oscillator example.}
\label{VanDerPol:TestErrorplot}
\end{figure}

\begin{figure}
\centering
\includegraphics[width=0.8\textwidth]{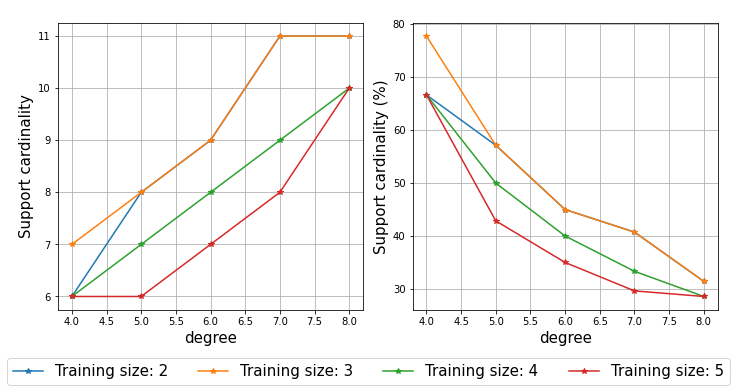}
\caption{Support cardinality of the solutions to the learning problems for modified Van der Pol oscillator example.}
\label{VanDerPol:SupportCard}
\end{figure}

%\begin{table}[h!]
%\centering
%\begin{tabular}{ |c||c|c|c|c|c|  }\hline
%Degree & 4 & 5 & 6 & 7 & 8\\  \hline
% $|X_n|$ & 9 & 14 & 20 & 27 & 35\\ \hline\end{tabular}
% \vspace{2mm} \hspace{-2.7cm}
%\caption{Cardinality of $X_n$}
%\label{VanDerPol:XCard}
%\end{table}

%\begin{figure}
%  %  \includegraphics[width=0.45\textwidth]{{"Graficos/VanDerPol/ObjScatter_degree_8_N_train_1"}.png}
%     \includegraphics[width=0.45\textwidth]{{"Graficos/VanDerPol/ObjScatter_degree_8_N_train_2"}.png}
% \includegraphics[width=0.45\textwidth]{{"Graficos/VanDerPol/ObjScatter_degree_8_N_train_5"}.png}
%\caption{Validation scatter.}
%
%\label{VanDerPol:Scatter}
%\end{figure}
%

%\begin{figure}[H]
\begin{figure}
\includegraphics[width=0.95\textwidth]{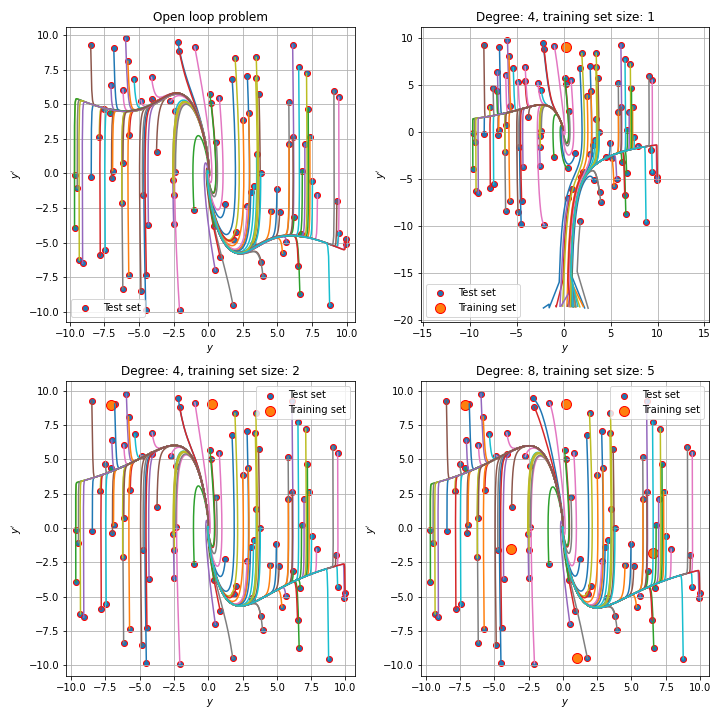}
\caption{Phase plane for test initial conditions for modified Van der Pol oscillator example.}
\label{VanDerPol:Phaseplane}
\end{figure}

\subsection{Allen-Cahn Equation.}
We turn to stabilization of  the Allen-Cahn equation with the  Neumann boundary conditions and consider
\beq
\begin{array}{c}
     \dis \min_{u_i\in L^{2}([0,T),\R)} \int_{0}^{T}\int_{-1}^{1}|y(x,t)|^{2}dxdt+\beta\int_{0}^{T}|u(t)|^{2}dt \\
     \ecart\dis y'(t,x)=\nu\frac{\partial^{2}y}{\partial x^{2}}(t,x)+y(t,x)(1-y^2(t,x))+\sum_{i=1}^{3}\chi_{\omega_i}(x)u_i(t)\\
     \ecart\dis\frac{\partial y}{\partial x}(t,-1)=\frac{\partial y}{\partial x}(t,1)=0, \quad  y(0,x)=y_0(x)
     \label{AllenCahn}
\end{array}
\eeq
for $x\in (-1,1)$ and $t>0$, where  $\nu=0.5$, $T=4$ and $\chi_{\omega_i}$ are the indicators functions of the sets $\omega_{1}=(-0.7,-0.4)$, $\omega_{2}=(-0.2,0.2)$, and $\omega_{3}=(0.4,0.7)$. This problem admits 3 steady states, which are $-1$, $0$ and $1$, with $0$ being unstable. \par
 Since problem \eqref{AllenCahn} is infinite-dimensional, we discretize it by using a Chebyshev spectral collocation method with $19$ degrees of freedom. The first integral in \eqref{AllenCahn}  is approximated by means of the Clenshaw-Curtis quadrature. For further details on the Chebyshev spectral collocation method and the Clenshaw-Curtis quadrature we refer to \cite[Chapters 6, 19]{Boyd}, and \cite[Chapters 7, 12, 13]{Trefethen}.\par
 Due to the high dimensionality of this problem, the evaluation of the feedback law is computationally expensive. In order to mitigate this  difficulty the hyperbolic cross technique is used for the construction of the basis. Further,  sparsity of solution,  can be influenced by the  penalty coefficient $\gamma$. With this in mind, we pay attention to the influence of  $\gamma$ on the sparsity of the solution and performance of the obtained feedback laws.
\par
For the results presented below, we choose $r=0.9$, $l=10$, $X=\mathcal{S}_6\setminus(\mathcal{B}_1\cup \mathcal{O}(\B_6))$, $v_0=0$,  and 10 different values for $\gamma$ which are listed in \Cref{AllCahn:fig:NComps}. The cardinality of $X$ is 350. We sample uniformly at random 5 and 100 initial conditions in ${\Omega}$ as training and test sets, respectively. We train progressively starting with $\gamma=10^{-1}$, for which   we use $0$ as initial guess for the value function. For the remaining $\gamma$ values initialization is done with the solution of the previous $\gamma$ value. The total training time was approx 2 hrs. In Table \ref{AllCahn:fig:NComps} the cardinality of the  supports are reported.  Clearly, the size the support increases as $\gamma$ decreases. \par

In \Cref{AllenCahn:Errorplot} we present the  normalized errors calculated for the  training and test sets for each chosen $\gamma$. We observe that in accordance with \Cref{ConvergenceTheo2} all the training errors decrease when $\gamma$ tends to 0. We observe the same behaviour for the test error. However, the test mean $L_2$ normalized error for the controls stops decreasing at about  $40\%$. % Further, when $\gamma$ is approximately $10^{-1}$ the errors on training and test sets are close, but when $\gamma$ is approaching 0 they start to differ.
\par

 In \Cref{AllenCahn:Scatter} we present the scatter plot between the value of the objective of the closed loop problem and the value obtained by our approach when  $\gamma=10^{-1}$ and $\gamma=10^{-6}$.  In the first scatter we see that the slope and the intercept of the regression line are around $0.55$ and $0.12$, respectively. Moreover we observe  a high dispersion of the point around the regression line. On the other hand, in the second scatter the regression line is closer to the identity line and the dispersion around it is clearly lower than in the first scatter.\par

Summarizing, in accordance with \eqref{ConvergenceTheo2} we see that all the error measures decrease when $\gamma$ goes to 0 in the training. The same is true for the error measures on the test set.

\begin{table}
\centering
\scalebox{0.8}{
\begin{tabular}{ |c||c|c|c|c|c|c|c|c|c|c|}\hline
$\gamma$& 1.0e-01 & 8.9e-02 & 7.8e-02 & 6.7e-02 & 5.6e-02 & 4.4e-02 & 3.3e-02 & 2.2e-02 & 1.1e-02 & 1.0e-06 \\ \hline
\makecell{Support \\ cardinality}
 & 10 & 11 & 14 & 14 & 21 & 22 & 22 & 23 & 24 & 28 \\
\hline\end{tabular}}
\vspace{0.5pt}
\caption{Support cardinality of the solution obtained for each $\gamma$ for Allen-Cahn equation example.}
\label{AllCahn:fig:NComps}
\end{table}
\begin{figure}
\centering
\includegraphics[width=0.9\textwidth]{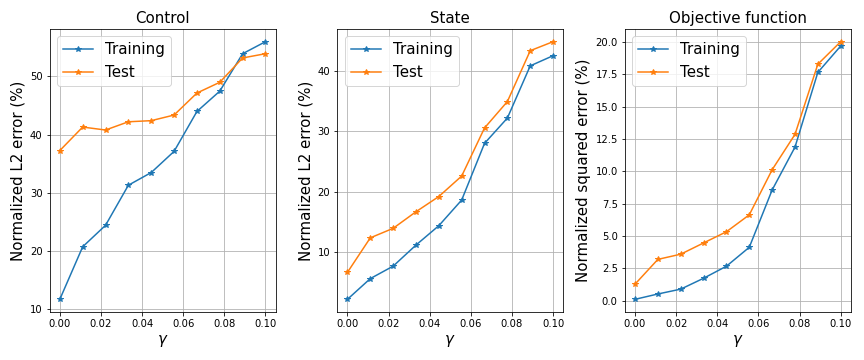}
\caption{Training and test mean normalized  errors for Allen-Cahn equation example.}
\label{AllenCahn:Errorplot}
\end{figure}
\begin{figure}
    \includegraphics[width=0.45\textwidth]{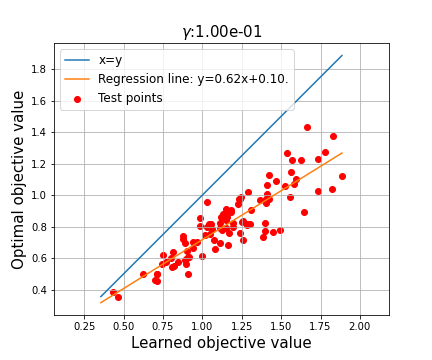}
     \includegraphics[width=0.45\textwidth]{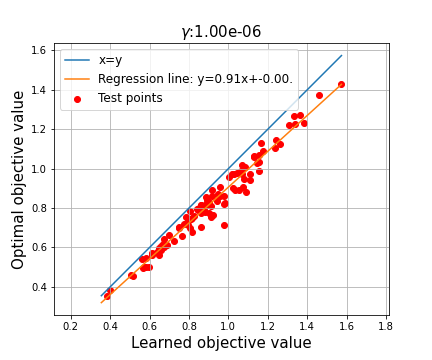}
\caption{Validation scatter for Allen-Cahn equation example.}
\label{AllenCahn:Scatter}
\end{figure}

 \subsection{Optimal Consensus for Cucker-Smale Model}
 We consider a set of $N$ agents with states $(x_i(t),y_i(t))\in\R^{2}\times\R^{2}$ for $i\in \{1,\ldots,N\}$ governed by the Cucker-Smale (see \citep{CuckerSmale}) dynamics. The system is controlled in such a way that the velocity of every agent asymptotically  approaches the mean mean velocity. In order to achieve this in an optimal sense, we solve (see \citep{Bailo,Camponigro})
  \beq  \begin{array}{l}
    \dis \min_{u_i\in L^{2}((0,\infty);\R^{2})} \frac{1}{N}\sum_{i=1}^{N}\int_{0}^{T}|y_i-\bar{y}|^{2}dt+\beta\sum_{i=1}^{N}\int_{0}^{T}|u_i|^{2}dt     \\
   \ecart \dis s.t.\quad x'_i=y_i,\quad y_i'=\frac{1}{N}\sum_{j=1}^{N}a(|x_i-x_j|)(y_j-y_i) +u_i,\\ \ecart \qquad \;\, \dis    x_i(0)=x^{i}_0,\quad y_i(0)=y^{i}_0,
  \end{array} \eeq
  where $a:[0,\infty)\to\R$ is a communication kernel given by $ a(r)=\frac{K}{(1+r^{2})}$ and $\bar{y}(t)$ is the mean velocity, that is
  $ \bar{y}(t)=\frac{1}{N}\sum_{j=1}^{N}y_j(t)$.
  \par
We set  $N=10$, $T=3$, $K=10^{-1}$, and $\beta=10^{-2}$ for the Cucker-Smale problem. For the learning problem we take $X=\S_4\setminus (\B_1\cup O(\S_4))$, $\gamma=10^{-5}$, $r=0.9$ and $l=5$. We sample uniformly at random in $\Omega$ a set of $5$ initial conditions as training set  and a set of 100 initial conditions as test set. We take $v_0(x_1,\ldots,x_{N},y_1,\ldots,y_{N})=10K\beta\sum_{i=1}^{N}|v_{i}|^{2}$ as initial guess, which ensures the boundness of the solutions for the closed loop problem \eqref{ClosedLoopProblem3}.

\par
We present the  errors for the training and tests phases in  \Cref{CuckerSmale:ErrorTable}. We also provide the scatter plot between the value of the open loop problem and the value obtained by our approach in \Cref{CuckerSmale:Scatter}. In all the cases the error is below $2\%$. Further, in the scatter plot the slope of the regression line is around $1$, the intercept is $0.14$, and there is only a small dispersion around it. Finally, we mention that the cardinality of $X$ is 650, while the cardinality of the support of the solution of the learning problem is $58$. Thus  we only need approximately $8.9\%$ of the elements of $X$ to achieve the error shown in \Cref{CuckerSmale:ErrorTable}.
\begin{figure}
\centering
\includegraphics[width=0.5\textwidth]{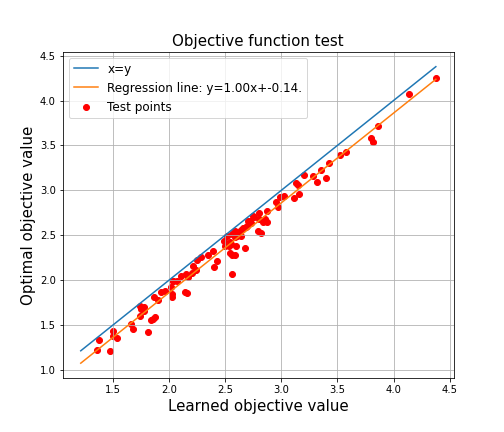}
\caption{Learned value vs optimal value on the test set for optimal consensus example.}
\label{CuckerSmale:Scatter}
\end{figure}
\begin{table}[h!]
\centering
 \begin{tabular}{ |c||c|c|c|  }\hline
& control error (\%)& state error (\%)& objective error (\%)\\ \hline
Training & 1.3090 & 0.7376 &  0.1932\\
Test & 0.1744 & 0.1032 &  0.4580\\
\hline\hline\end{tabular}
\vspace{3pt}
\caption{$SSE_u, SSE_y, SSE_J$ in percent for optimal consensus example.}
\label{CuckerSmale:ErrorTable}
\end{table}

\section{Conclusion}
A learning based  method to obtain feedback laws for nonlinear optimal control problems and their approximation by polynomials was presented. The proposed methodology was
implemented  in python and  tested on 4  problems. These experiments demonstrate the efficiency of the approach for  obtaining approximative   feedback laws for   non-linear and high dimensional problems. For the linear problem that we tested,  our approach was capable of finding feedback laws close to those provided by  the Riccati synthesis. Of course, this relates to the fact that for linear-quadratic problems the value function is a quadratic polynomial, and it is thus contained in our ansatz space.

% Acknowledgements should go at the end, before appendices and references

%\acks{We would like to acknowledge support for this project
%from the National Science Foundation (NSF grant IIS-9988642)
%and the Multidisciplinary Research Program of the Department
%of Defense (MURI N00014-00-1-0637). }


%\newpage

\appendix
\section*{Appendix}
\label{Apendice}
\begin{lemma}
\label{TimeExistenceLemma}
 Let  $\varepsilon\in (0,1)$, $\sigma\in (0,l)$, and $v$ and $v_{\varepsilon}$ be functions in $C^{1,1}(\overline{\Omega})$ satisfying
\beq \label{eq10.1}
 \left|B^{\top}(\nabla v-\nabla v_{\varepsilon})\right|_{C(\overline{\Omega})}< \varepsilon,  \mbox{ and }\mathcal{J}_{\infty}(v)<\infty. \eeq
Further assume that
\beq \label{TimeExistenceLemma:LInftyBound}
 |y_i(t)|\leq l-\sigma, \text{ for } i=1,\dots, I, \; \forall t \ge 0,
\eeq
where $y_{i}$ is the solution of \eqref{ClosedLoopProblem3}, and define
\beq C=2\left|f-\frac{1}{\beta}BB^{\top}\nabla v\right|_{Lip(\overline{\Omega})}+\frac{|B|^2}{\beta^{2}}. \label{TimeExistenceLemma:C} \eeq
Then, there exists $T_{\varepsilon}\in (0,\infty]$ satisfying
\beq
T_{\varepsilon}\geq \frac{1}{C}\ln\left(1+\frac{\sigma^{2} C}{4\varepsilon^{2}}\right),
\label{TimeExistenceLemma:Teps}
\eeq
such that for each $i\in\{1,\ldots,I\}$ problem \eqref{ClosedLoopProblem3} admits a solution $y_i^{\varepsilon}\in C^{1}([0,T_{\varepsilon}];\Omega)$, and
\beq |y^{\varepsilon}_{i}(t)-y_{i}(t)|^{2}\leq \frac{\varepsilon^2}{C}(e^{Ct}-1) \mbox{ and }|y^{\varepsilon}_{i}(t)|\leq l-\frac{\sigma}{2}\mbox{ for all }t\in [0,T_{\varepsilon}].\label{TimeExistenceLemma:conv} \eeq
Moreover, defining $\tilde{T}_{\varepsilon}>0$ by
\beq  \tilde{T}_{\varepsilon}:=\frac{1}{C}\ln\left( 1+\frac{C\sigma^{2}}{4\varepsilon^{1/2}} \right) \label{TimeExistenceLemma:Teps2}\eeq
we have
\beq  \left|\mathcal{J}_{\tilde{T}_{\varepsilon}}(v_{\varepsilon})-\mathcal{J}_{\tilde{T}_{\varepsilon}}(v)\right|\leq K \varepsilon^{1/4}, \label{TimeExistenceLemma:JConver}\eeq
where $K$ is a constant independent of $\varepsilon$.
\end{lemma}
\begin{proof}[Proof of \Cref{TimeExistenceLemma}]
Since $f$ is Lipschitz on bounded sets,  there exists a time $T_{\varepsilon}>0$ such that for each $i\in\{1,\ldots,I\}$ the closed loop problem \eqref{ClosedLoopProblem3} admits a solution $y_{i}^{\varepsilon}\in C([0,T_{\varepsilon}];\R^{d})$ where $T_{\varepsilon}\in (0,\infty]$ is defined as the largest time such that
\beq \sup_{\begin{array}{c}
     i\in \{1,\ldots,I\},  \\
     t\in [0,T_{\varepsilon}].
\end{array}}|y_i^{\varepsilon}(t)|= l-\frac{\sigma}{2}.\label{TEpsDef}\eeq
We next verify  that $T_{\varepsilon}$ satisfies \eqref{TimeExistenceLemma:Teps}
 where $C$ is given by \eqref{TimeExistenceLemma:C} and $\sigma\in (0,l)$ satisfies \eqref{TimeExistenceLemma:LInftyBound}. We assume that $T_{\varepsilon}$ is finite, otherwise this claim is trivially satisfied. Let $y_{i}$ be the solution of \eqref{ClosedLoopProblem3} for $v$. Then, subtracting  the equations \eqref{ClosedLoopProblem3} for $y^{\varepsilon}_{i}$ and $y_{i}$ we obtain
\begin{equation}
    (y^{\varepsilon}_{i}-y_{i})'=(f(y^{\varepsilon}_{i})-f(y_{i}))-\frac{1}{\beta}BB^{\top}\left(\nabla v_{\varepsilon}(y^{\varepsilon}_{i})-\nabla v(y_{i})\right) \mbox{ in }(0,T_{\varepsilon}),
    \label{TimeExistenceLemma:eq1}
\end{equation}
and $y_{i}^{\varepsilon}(0)-y_{i}(0)=0$. We have
\beq
\nabla v_{\varepsilon}(y^{\varepsilon}_{i})-\nabla v(y_{i})
=\left(\nabla v_{\varepsilon}(y^{\varepsilon}_{i})-\nabla v(y^{\varepsilon}_{i})\right)+\left(\nabla v(y^{\varepsilon}_{i})-\nabla v(y_{i})\right)
\label{TimeExistenceLemma:eq2}
\eeq
 in $(0,T_{\varepsilon})$. Multiplying \eqref{TimeExistenceLemma:eq1} by $y^{\varepsilon}_{i}-y_{i}$ and using \eqref{eq10.1}, \eqref{TimeExistenceLemma:C} and \eqref{TimeExistenceLemma:eq2} we get
\begin{equation}
    \frac{d}{dt}\left(\frac{1}{2}|y^{\varepsilon}_{i}-y_{i}|^{2}\right)\leq \frac{C}{2}|y^{\varepsilon}_{i}-y_{i}|^{2}+\frac{\varepsilon^{2}}{2}\mbox{ in }(0,T_{\varepsilon}).
    \label{TimeExistenceLemma:eq3}
\end{equation}
Multiplying both sides of \eqref{TimeExistenceLemma:eq3} by $e^{-Ct}$ and integrating between $0$ and $t$ we obtain for each $i\in\{1,\ldots,I\}$
\begin{equation} |y^{\varepsilon}_{i}-y_{i}|^{2}\leq \frac{\varepsilon^{2}}{C}\left(e^{Ct}-1\right)\mbox{ in }[0,T_{\varepsilon}],
 \label{TimeExistenceLemma:eq4}
\end{equation}
and \eqref{TimeExistenceLemma:conv} holds.  By \eqref{TimeExistenceLemma:eq4} and \eqref{TimeExistenceLemma:LInftyBound} we have
\beq  |y_{i}^{\varepsilon}|_{\infty}\leq |y_{i}^{\varepsilon}-y_i|_{\infty}+|y_{i}|_{\infty} \leq \frac{\varepsilon}{C^{1/2}}\left(e^{Ct}-1\right)^{1/2}+(l-\sigma) \mbox{ in }[0,T_{\varepsilon}].\eeq
Combining \eqref{TEpsDef} and the previous inequality
we obtain
$$ \frac{\sigma}{2}\leq  \frac{\varepsilon}{C^{1/2}}\left(e^{CT_{\varepsilon}}-1\right)^{1/2} $$
which clearly implies \eqref{TimeExistenceLemma:Teps} . \par
Now we turn to the proof of \eqref{TimeExistenceLemma:JConver}. Since $\varepsilon\in (0,1)$ and recalling the definition of $\tilde{T}_{\varepsilon}$ we have $$ \tilde{T_{\varepsilon}}=\frac{1}{C}\ln\left(1+\frac{C\sigma^{2}}{4\varepsilon^{1/2}}\right) <\frac{1}{C}\ln\left(1+\frac{C\sigma^{2}}{4\varepsilon^{2}}\right)\leq T_{\varepsilon}.$$
%Therefore, we have $\tilde{T}_{\varepsilon}<T_{\varepsilon},$ which implies that \eqref{TimeExistenceLemma:conv} holds for $t\in (0,\tilde{T}_{\varepsilon}]$. By \eqref{TimeExistenceLemma:Teps2} and using that $\ln(x+1)\leq x$ for all $x\geq 0$ we get
%\beq  \tilde{T}_{\varepsilon}\leq \frac{\sigma^{2}}{4\varepsilon^{1/2}}. \label{BodunTildeTEps} \eeq
Therefore, we have $\tilde{T}_{\varepsilon}<T_{\varepsilon},$ and using that $\ln(x+1)\leq x$ for all $x\geq 0$ we also get $\tilde{T}_{\varepsilon}\leq \frac{\sigma^{2}}{4\varepsilon^{1/2}}.$
Since $\ell$ is $C^{1}$, and $\ell(0)=0$  we can estimate
\beq \left|\int_{0}^{\tilde{T}_{\varepsilon}}\ell(y_{i}^{\varepsilon})dt-\int_{0}^{\tilde{T}_{\varepsilon}}\ell(y_{i})dt\right|\leq
    |\ell|_{Lip(\overline{\Omega})}\int_{0}^{\tilde{T}_{\varepsilon}}|y_{i}^{\varepsilon}-y_{i}|dt.\eeq
Together with \eqref{TimeExistenceLemma:eq4} this implies that
\beq \left|\int_{0}^{\tilde{T}_{\varepsilon}}\ell(y_{i}^{\varepsilon})dt-\int_{0}^{\tilde{T}_{\varepsilon}}\ell(y_{i})dt\right|\leq
    \frac{\varepsilon}{C^{1/2}} \tilde{T}_{\varepsilon}\left(e^{C\tilde{T}_{\varepsilon}}-1\right)^{1/2}|\ell|_{Lip(\overline{\Omega})}. \eeq
By the definition of $\tilde T_{\varepsilon}$ and since $\tilde{T}_{\varepsilon}\leq \frac{\sigma^{2}}{4\varepsilon^{1/2}}$, we obtain
\beq \left|\int_{0}^{\tilde{T}_{\varepsilon}}\ell(y_{i}^{\varepsilon})dt-\int_{0}^{\tilde{T}_{\varepsilon}}\ell(y_{i})dt\right|\leq
    \varepsilon^{1/4}\frac{\sigma^{3}}{8} |\ell|_{Lip(\overline{\Omega})}.\label{ConvergenceTheo2Eq7:1} \eeq

To estimate the second summand in $\mathcal{J}_{\tilde{T}_{\varepsilon}}(v_{\varepsilon})-\mathcal{J}_{\tilde{T}_{\varepsilon}}(v)$ we first observe that due to \eqref{eq10.1}, \eqref{TimeExistenceLemma:LInftyBound}, and \eqref{TimeExistenceLemma:conv} there exists a constant $K$ independent of $\varepsilon\in (0,1)$ such that
for all $i \in \{1,\ldots,I\}$
$$\max_{t\in [0,\tilde{T}_{\varepsilon}]}\left|B^{\top}\left(\nabla v_{\varepsilon}(y_{i}^{\varepsilon}(t))+\nabla v(y_{i}(t))\right)\right|\leq K.$$ Consequently we find by
\eqref{eq10.1} and \eqref{TimeExistenceLemma:conv}
\beq \label{eq:11.17}
\begin{array}{l}
\dis \big| \int_{0}^{\tilde{T}_{\varepsilon}}|B^{\top}\nabla v_{\varepsilon}(y_i^{\varepsilon})|^{2}dt-\int_{0}^{\tilde{T}_{\varepsilon}}|B^{\top}\nabla v(y_i)|^{2}dt\, \big| \\[1.5ex]
\le\dis \int_{0}^{\tilde{T}_{\varepsilon}}|B^{\top}\nabla (v_{\varepsilon}(y_i^{\varepsilon}) + v(y_i))| \, |B^{\top}\nabla (v_{\varepsilon}(y_i^{\varepsilon}) - v(y_i))| \, dt
\\[1.5ex]
\le\dis K \int_{0}^{\tilde{T}_{\varepsilon}}|B^{\top}\nabla (v_{\varepsilon}(y_i^{\varepsilon}) - v(y_i^{\varepsilon}))| +  |B^{\top}\nabla (v(y_i^{\varepsilon}) - v(y_i))| \,dt
\\[1.5ex]
\le \dis K(\varepsilon \tilde T_{\varepsilon} +|B^\top \nabla v|_{Lip(\bar \Omega)} \int_{0}^{\tilde{T}_{\varepsilon}} |y_i^\varepsilon -y_i| \, dt)\\[1.5ex]
\le \dis K(\varepsilon \tilde T_{\varepsilon} +|B^\top \nabla v|_{Lip(\bar \Omega)} \tilde T_{\varepsilon} \frac{\varepsilon}{\sqrt {C}} \sqrt{e^{C\tilde T_\varepsilon} - 1}) \le K [ \frac{{\varepsilon}^{\frac{1}{2}} \sigma^2}{4} + |B^\top \nabla v|_{Lip(\bar \Omega)}\frac{\varepsilon^{\frac{1}{4}}  \sigma^3}{8} ].
\end{array}
\eeq
Inequality \eqref{TimeExistenceLemma:JConver} is obtained from \eqref{ConvergenceTheo2Eq7:1} and \eqref{eq:11.17}.
\end{proof}

\begin{lemma}
Consider $T\in (0,\infty]$ and a sequence $v_{k}\in C^{1,1}(\overline{\Omega})$ converging  in $C^{1,1}(\overline{\Omega})$ to $v$, such that $\mathcal{J}_{T}(v_{k})<\infty$ and $\mathcal{J}_{T}(v)<\infty$ .  Then we have
\begin{equation}
    \lim_{k\to\infty}\mathcal{J}_{T}(v_{k})=\mathcal{J}_{T}(v),
    \label{JTContinuity}
\end{equation}
for $T\in (0,\infty)$ and otherwise
\begin{equation}
    \mathcal{J}_{\infty}(v)\leq \liminf_{k\to\infty}\mathcal{J}_{\infty}(v_{k}).
    \label{JSemiContinuity}
\end{equation}
\label{ContinuityLemma}
\end{lemma}
\begin{proof}[Proof of \Cref{ContinuityLemma}]
Consider first $T\in (0,\infty)$, $y_{i}^{k}$ and $y_{i}$ the solutions of the closed loop problems \eqref{ClosedLoopProblem2} for $v_{k}$ and $v_{*}$ respectively. Recalling the definition of $\mathcal{J}_{T}$ in \eqref{eq:Jdef} and the assumption that $\mathcal{J}_{T}(v_k)<\infty$, we know that $|y_{i}^{k}(t)|\leq l$ for all $t\in [0,T]$, $i\in \{1,\ldots,I\}$, and $k\in\{1,2,\ldots\}$. By the Lipschitz continuity of $f$ on $\overline{\Omega}$ and \eqref{ClosedLoopProblem2}, we get that the set $\{y_{i}^{k}:\ i=1,\ldots,I;\ k=1,2,\ldots,\}$ is bounded in $H^{1}((0,T);\R^{d})$. Therefore, for every $i\in \{1,\ldots,I\}$ there exists a function $\bar{y}_{i}\in H^{1}((0,T);\R^{d})\cap L^{\infty}((0,T);\R^{d})$ such that, passing to a sub-sequence,
\beq y_{i}^{k}\weak \bar{y}_{i} \mbox{ in }H^{1}((0,T);\R^{d})\mbox{ and }y_{i}^{k}\to \bar{y}_{i} \mbox{ in }C([0,T];\R^{d}).\eeq
Further, as $v_{k}$  converges to $v$ in $C^{1,1}(\overline{\Omega})$ and $y_{i}^{k}$  converges to $\bar{y}$ in $C([0,T])$, we get
\begin{equation}
    \lim_{k\to\infty}\nabla v_{k}(y^{k}_{i})=\nabla v(\bar{y}_{i}) \mbox{ in }C([0,T]) \mbox{ for all }i\in \{1,\ldots,I\}
    \label{ContinuityLemma:eq1}
\end{equation}
and
\begin{equation}
    \lim_{k\to\infty}\left\{f(y_{i}^{k})-\frac{1}{\beta}BB^{\top}\nabla v_{k}(y^{k}_{i})\right\}=f(\bar{y}_{i})-\frac{1}{\beta}BB^{\top}\nabla v(\bar{y}_{i}) \mbox{ in }C([0,T]),
\end{equation}
 for all $i\in \{1,\ldots,I\}$. This implies that the functions $\{\bar{y}_{i}\}_{i=1}^{I}$ are solutions of \eqref{ClosedLoopProblem2} and by uniqueness of the solutions of this problem, we have
$$\bar{y}_{i}=y_{i},\ \mbox{ for all }i\in \{1,\ldots,I\}.$$
Hence we obtain
\beq
\lim_{k\to\infty} y_{i}^{k}= y_{i} \mbox{ in }C^{1}([0,T];\R^{d}) \label{ContinuityLemma:eq2} \mbox{ for all }i\in \{1,\ldots,I\}.
\eeq
\par
By the continuity of $\ell$, \eqref{ContinuityLemma:eq1} and \eqref{ContinuityLemma:eq2}, we get that \eqref{JTContinuity} is verified for $T\in (0,\infty)$.
For $T=\infty$, we find that \eqref{ContinuityLemma:eq2} holds for all $\bar T \in (0,\infty)$.
Since $\ell$ is bounded from below by 0, we have $$ \mathcal{J}_{\bar{T}}(v_{k})\leq \mathcal{J}_{\infty}(v_{k}), \mbox{ for every }\bar{T}\in (0,\infty).$$
Then, taking the limit inf when $k\to\infty$ on both sides of the previous inequality  we get
$$\mathcal{J}_{\bar{T}}(v)\leq \liminf_{k\to\infty} \mathcal{J}_{\infty}(v_{k}), \mbox{ for every }\bar{T}\in (0,\infty), $$
where we use that $\liminf = \lim $ on the left hand side.
Finally, taking $\bar{T}\to\infty$ we obtain \eqref{JSemiContinuity}.
\end{proof}
\begin{proof}[Proof of \Cref{ExistenceTheo}]
Assume that there exists a feasible solution of problem \eqref{PolyLearningProblem}. Since the objective function is bounded from below, there exists an infimizing sequence $\theta^{k}\in \R^{M}$. We denote the infimum of \eqref{PolyLearningProblem} by $\mathcal{J}^{*}_{T}$ and set $v_{k}=\sum_{i=1}^{M}\theta^{k}_i\phi_i$, where $\theta^{k}_i$ is the $i$-th component of $\theta^{k}$. Since $\theta^{k}$ is an infimizing sequence, the sequence $\{P_{\gamma,r}(\theta^{k})\}_{k\in\N}$ is bounded, that is
$$ P_{\gamma,r}(\theta^{k})=\gamma\left( \frac{(1-r)}{2} |\theta^{k}|_{2}^{2}+r|\theta^{k}|_{1}\right)\leq C,$$
for some $C>0$ independent of $k$. This implies that there exists $\theta^{*}\in \R^{M}$ such that, passing to a sub-sequence
$$ \theta^{k}\to \theta^{*}\mbox{ and } v_{k}\to v^{*}=\sum_{i=1}^{n}\theta_{i}^{*}\phi_{i}\mbox{ in }C^{1,1}(\overline{\Omega}).$$
Hence, by Lemma \ref {ContinuityLemma} we have that
$ \tilde{\mathcal{J}}_{T}(\theta^{*})+ P_{\gamma,r}(\theta^{*})\leq \mathcal{J}^{*}_T,$
and we conclude that $\theta^{*}$ is a solution of \eqref{PolyLearningProblem}.
\end{proof}

\begin{proof}[Proof of \Cref{ConvergenceTheo2}]
We provide the proof in several steps.\par
{\em Step 1.} Let $V$ be the value function of \eqref{ControlProblem}. Since $V$ is assumed to be  $C^{1,1}(\overline{\Omega})$, by Theorem 9 in \cite[Section 7.2]{Hayek}, for every $\varepsilon\in (0,1)$ there exists a natural number $k(\varepsilon)$ and a polynomial $V_{\varepsilon}\in \P_{k(\varepsilon)}$ such that

%\todo[inline] { Total degree vs. Hyperbolic cross !! Define the X}

\beq  \norm{B^\top(\nabla V_{\varepsilon}-\nabla V)}_{C(\overline{\Omega})}< \varepsilon.\label{ConvergenceV}\eeq
Since $\nabla V(0)=0$ and $V(0)=0$, we assume that $V_{\varepsilon}(0)=0$ and $\nabla V_{\varepsilon}(0)=0$, otherwise we can redefine it by subtracting $V_{\varepsilon}(0)+\nabla V_{\varepsilon}(0)\cdot x$ from it. We denote the coefficients of $V_{\varepsilon}$ with respect to the basis $X_{k(\varepsilon)}=\B_{k(\varepsilon)}\setminus \B_1$ by $\theta^{\varepsilon}\in \R^{M_{k(\varepsilon)}}$. By \Cref{TimeExistenceLemma}, for  every $i\in \{1,\ldots,I\}$ problem \eqref{ClosedLoopProblem3} has a solution $y^{\varepsilon}_{i}\in C^{1}([0,\tilde T_{\varepsilon}],\overline{\Omega})$, with $T=\tilde{T}_{\varepsilon}$ and $v=V_{\varepsilon}$, where $\tilde{T}_{\varepsilon}$ is given by \eqref{TimeExistenceLemma:Teps2}. Moreover, by \eqref{TimeExistenceLemma:JConver} we know that there exists $K>0$ independent of $\varepsilon$, such that
\beq
\mathcal{J}_{\tilde{T}_\varepsilon}(V_{\varepsilon})\leq \mathcal{J}_{\tilde{T}_\varepsilon}(V)+K\varepsilon^{1/4}.
\label{ConvergenceTheo2:proof:JConver}
\eeq

%\todo[inline]{local lip vs. lip on bounded sets}

{\em Step 2.}
 We consider problem \eqref{PolyLearningProblem} with $X=\B_{k(\varepsilon)} \setminus \B_1$, $T=\tilde{T}_{\varepsilon}$,
 $\gamma>0$ and $r\in [0,1]$,  and denote its  solution   by $\theta^{\gamma,r,k(\varepsilon),\tilde{T}_{\varepsilon}}$, which we know to exist by Lemma \ref{TimeExistenceLemma} and Theorem \ref{ExistenceTheo}.
 We point out that it is possible to use $X=\S_{\tilde{k}(\varepsilon)}\setminus \B_1$ instead of $B_{k(\varepsilon)}$, provided that $\tilde{k}(\varepsilon)$ is sufficiently large  such that $\B_{k(\varepsilon)}\subset\S_{\tilde{k}(\varepsilon)}$. We set $$v_{\gamma,r,k(\varepsilon),\tilde{T}_{\varepsilon}}=\sum_{i=1}^{M_{k(\varepsilon)}}\theta^{\gamma,r,k(\varepsilon),\tilde{T}_{\varepsilon}}_i \phi_{i}.$$
Since $\theta^{\gamma,r,k(\varepsilon),\tilde{T}_{\varepsilon}}$ is optimal, we have
\begin{equation}
    \mathcal{J}_{\tilde{T}_{\varepsilon}}(v_{\gamma,k(\varepsilon),\tilde{T}_{\varepsilon}})+P_{\gamma,r}(\theta^{\gamma,r,k(\varepsilon),\tilde{T}_{\varepsilon}})\leq \mathcal{J}_{\tilde{T}_{\varepsilon}}(V_{\varepsilon})+P_{\gamma,r}(\theta^{\varepsilon})
    \label{ConvergenceTheo2Eq52}
\end{equation}
and by \eqref{ConvergenceTheo2:proof:JConver}
we obtain
\begin{equation}
    \mathcal{J}_{\tilde{T}_\varepsilon}(v_{\gamma,k(\varepsilon),\tilde{T}_{\varepsilon}})+P_{\gamma,r}(\theta^{\gamma,r,k(\varepsilon),\tilde{T}_{\varepsilon}})\leq \mathcal{J}_{\infty}(V)+K \varepsilon^{1/4}+P_{\gamma,r}(\theta^{\varepsilon}).
    \label{ConvergenceTheo2Eq8}
\end{equation}
We now choose $\gamma=\gamma_{\varepsilon}$ such that $  P_{\gamma_{\varepsilon},r}(\theta^{\varepsilon})= K\varepsilon^{1/4}.$ Then, we obtain
  \beq \mathcal{J}_{\tilde{T}_{\varepsilon}}(v_{\gamma_{\varepsilon},r,n(\varepsilon),\tilde{T}_\varepsilon})+P_{\gamma_{\varepsilon},r}(\theta^{\gamma_{\varepsilon},r,k(\varepsilon),T_{\varepsilon}})\leq \mathcal{J}_{\infty}(V)+2K\varepsilon^{1/4} \label{BoundJK}\eeq
  and taking $\varepsilon\to 0$, we get for every $r\in [0,1]$
 \beq \limsup_{\varepsilon\to 0} \mathcal{J}_{\tilde{T}_{\varepsilon}}(v_{\gamma
 _\varepsilon,r,n(\varepsilon),\tilde{T}_\varepsilon})\leq \mathcal{J}_{\infty}(V). \label{ConvergenceTheo2Eq9}\eeq
 \par
{\em Step 3.}  For $i\in\{1,\ldots,I\}$, we denote the solutions of \eqref{ClosedLoopProblem3}, for $y_0=y^{i}_0$, $T=\tilde{T}_{\varepsilon}$, and $v=v^{\gamma_{\varepsilon},r,k(\varepsilon),\tilde{T}_{\varepsilon}}$, by $y_i^{\gamma_{\varepsilon},r,k(\varepsilon),\tilde{T}_{\varepsilon}}$, and we define the controls $u_{i}^{\gamma_{\varepsilon},r,k(\varepsilon),\tilde{T}_{\varepsilon}}\in L^{2}((0,\tilde{T}_\varepsilon);\R^{m})$ by
$$u_{i}^{\gamma_{\varepsilon},r,k(\varepsilon),\tilde{T}_{\varepsilon}}(t)=-\frac{1}{\beta}B^{\top}\nabla v^{\gamma_{\varepsilon},r,k(\varepsilon),\tilde{T}_{\varepsilon}}(y_{i}^{\gamma_{\varepsilon},r,k(\varepsilon),\tilde{T}_{\varepsilon}}(t))\mbox{ in }(0,\tilde{T}_\varepsilon).$$
Now, by the definition of the controls we have for $i\in \{1,\ldots,I\}$ and all $\bar{T}\in (0,\tilde{T}_{\varepsilon})$ $$ \int_{0}^{\bar{T}}|u_i^{\gamma_{\varepsilon},r,k(\varepsilon),\tilde{T}_{\varepsilon}}|^2 dt=\int_{0}^{\bar{T}} \left|B^{\top}\nabla v^{\gamma_{\varepsilon},r,k(\varepsilon),\tilde{T}_{\varepsilon}}(y_i^{\gamma_{\varepsilon},r,k(\varepsilon),\tilde{T}_{\varepsilon}})\right|^{2}dt\leq 2\beta I|B|^{2}\tilde{\mathcal{J}}_{\tilde{T}_{\varepsilon}}(v^{\gamma_{\varepsilon},r,k(\varepsilon),\tilde{T}_{\varepsilon}}).$$
Using \eqref{BoundJK} and $\varepsilon\le 1$ in the previous inequality we get
\beq \int_{0}^{\bar{T}}|u_i^{\gamma_{\varepsilon},r,k(\varepsilon),\tilde{T}_{\varepsilon}}|^2 dt\leq 2\beta I|B|^{2}(2K+\mathcal{J}_{\infty}(V)).\label{ControlBound}\eeq
In virtue of \eqref{ControlBound}, \eqref{ClosedLoopProblem3},  and the fact that $|y_i^{\gamma_{\varepsilon},r,k(\varepsilon),\tilde{T}_{\varepsilon}}(t)|\leq l$ for all $t\in [0,\tilde{T}_{\varepsilon}]$, we get
\beq \int_{0}^{\bar{T}}\left|\frac{d}{dt}y_i^{\gamma_{\varepsilon},r,k(\varepsilon),\tilde{T}_{\varepsilon}}(t)\right|^{2}dt\leq \bar{T}\sup_{x\in\overline{\Om}}|f(x)|^{2} +2\beta I|B|^{2}(2K+\mathcal{J}_{\infty}(V)).\eeq
Thus, for every $\bar{T}\in (0,\infty)$, $i=1,\ldots,I$ and taking $\varepsilon\to 0$, there exist $y_{i}^{*}\in H^{1}_{loc}((0,\infty);\R^{d})$ and $u_{i}^{*}\in L^{2}_{loc}((0,\infty);\R^{m})$ such that, passing to a sub-sequence
\beq
y_{i}^{\gamma_{\varepsilon},r,k(\varepsilon),\tilde{T}_{\varepsilon}}\weak y_{i}^{*} \mbox{ in }H^{1}((0,\bar{T});\R^{d})\mbox{ and }u_i^{\gamma_{\varepsilon},r,k(\varepsilon),\tilde{T}_{\varepsilon}}\weak u_i^{*}\mbox{ in }L^2((0,\bar{T});\R^{m}).\label{LimYU}
\eeq
Further, by the compact inclusion of $C([0,\bar{T}];\R^{d})$ into $H^{1}((0,\bar{T}),\R^{d})$, we have
\beq y_{i}^{\gamma_{\varepsilon},r,k(\varepsilon),\tilde{T}_{\varepsilon}}\to y_{i}^{*} \mbox{ in }C([0,\bar{T}];\R^{d})\label{LimYC}\eeq
when $\varepsilon\to 0$, for every $\bar{T}\in (0,\infty)$.\par
For $i\in \{1,\ldots,d\}$, we use \eqref{LimYU}, \eqref{LimYC} and take $\varepsilon\to0$ in \eqref{ClosedLoopProblem3} to obtain
\beq  (y^{*}_{i})'(t)= f(y_{i}^{*}(t))+Bu_{i}^{*}(t),\ \forall \ t \in (0,\infty),\quad  y_i^{*}(0)=y_{0}^{i}.\label{ConvergenceTheo2Eq10}\eeq
Additionally, using the definitions of $y_{i}^{\gamma_{\varepsilon},r,k(\varepsilon),T_{\varepsilon}}$ and $u_{i}^{\gamma_{\varepsilon},r,k(\varepsilon),T_{\varepsilon}}$ for $i\in \{1,\ldots,I\}$ together with \eqref{LimYU}, \eqref{LimYC}, and  the lower semi-continuity of $|\cdot|^{2}$ we have
\beq \frac{1}{I}\sum_{i=1}^{I}\int_{0}^{\bar{T}}\ell(y^{*}_{i})dt+\frac{\beta}{2}\int_{0}^{\bar{T}}|u_{i}^{*}|^{2}dt\leq \liminf_{\varepsilon\to 0} \mathcal{J}_{\tilde{T}_{\varepsilon}}(v_{\gamma
 _\varepsilon,r,n(\varepsilon),\tilde{T}_\varepsilon})\ \forall \bar{T} \in (0,\infty).\label{BoundControlProblem}\eeq
In particular, since $\bar{T}\in (0,\infty)$ in \eqref{BoundControlProblem} is arbitrary, we get
\beq \frac{1}{I}\sum_{i=1}^{I}\int_{0}^{\infty}\ell(y^{*}_{i})dt+\frac{\beta}{2}\int_{0}^{\infty}|u_{i}^{*}|^{2}dt\leq \liminf_{\varepsilon\to 0} \mathcal{J}_{\tilde{T}_{\varepsilon}}(v_{\gamma
 _\varepsilon,r,n(\varepsilon),\tilde{T}_\varepsilon}).\label{ConvergenceTheo2Eq11}\eeq
By \eqref{ConvergenceTheo2Eq10}, \eqref{ConvergenceTheo2Eq11} and the definition of the value function we have
\beq\begin{array}{l}
\dis\mathcal{J}_{\infty}(V)=\frac{1}{I}\sum_{i=1}^{I}\int_{0}^{\infty}\ell(y_{i})dt+\frac{\beta}{2}\int_{0}^{\infty}|B^{\top}\nabla V(y_{i})|^{2}dt\\
 \ecart\dis\leq \frac{1}{I}\sum_{i=1}^{I}\int_{0}^{\infty}\ell(y^{*}_{i})dt+\frac{\beta}{2}\int_{0}^{\infty}|u_{i}^{*}|^{2}dt\leq \liminf_{\varepsilon\to 0} \mathcal{J}_{\tilde{T}_{\varepsilon}}(v_{\gamma
 _\varepsilon,r,n(\varepsilon),\tilde{T}_\varepsilon}).
\end{array} \label{ConvergenceTheo2Eq12}\eeq
Finally, \eqref{ConvergenceTheo2Eq9} and \eqref{ConvergenceTheo2Eq12} imply
 \beq \lim_{\varepsilon\to 0} \mathcal{J}_{\tilde{T}_{\varepsilon}}(v_{\gamma
 _\varepsilon,r,n(\varepsilon),\tilde{T}_\varepsilon})= \mathcal{J}_{\infty}(V). \label{ConvergenceTheo2Eq13}\eeq
 which concludes the proof.
\end{proof}
\begin{prop}
\label{ExistenceProp} Assume that  $\nu\in C^{2}(\overline{\Omega})$ with $\nabla \nu(0)=0$ and $\sigma\in(0,l) $ are such that  \eqref{ClosedLoopProblem3} with $v=\nu$,  $T= \infty$  and $i\in \{1,\dots, I\}$ admits a solution $y_i$, satisfying
\beq \lim_{t\to \infty}y_i(t)=0, \ |y_i(t)|_{\infty}\leq l-\sigma,\ \forall t\in [0,\infty), \mbox{ for all }i\in\{1,\ldots,I\}. \label{ExistencePropStability}\eeq
Suppose further that the linearized system
\beq  z'=\left(Df(0)-\frac{1}{\beta}BB^{\top}\nabla \nu^{2}(0)\right)z,\quad z(0)=z_0 \label{ExistenceProp:LinearizedSys} \eeq
is exponentially stable, i.e. there exist $C>0$ and $\mu>0$ such that
$$|z|\leq Ce^{-\mu t}|z_0|\mbox{ for all } t \in (0,\infty)\mbox{ and }z_0\in \R^{d}.$$
Then, there exist $\varepsilon_0 \in (0,1)$, $\rho>0$, $K>0$, and $\kappa>0$, such that for every $\tilde{\nu}\in C^{2}(\overline{\Omega})$ which satisfies
\beq \label{eq:11.41}
 \norm{\nu-\tilde{\nu}}_{C^{2}(\overline{\Omega})}\leq \varepsilon_0\mbox{ and }\nabla\tilde{\nu}(0)=0,
\eeq
we have that the closed loop system \eqref{ClosedLoopProblem3} with $v=\tilde{\nu}$ is exponentially stable for every $y_0\in B(0,\rho)$, and for every $i\in\{1,\ldots,I\}$
\beq  |\tilde{y}_{i}|\leq K e^{-\kappa t}|y_0^i|\mbox{ for all }t \in (0,\infty),\label{ExistenceProp:ExpStabVesp}\eeq
and $\mathcal{J}_{\infty}(\tilde{\nu})<\infty$, where $\{\tilde{y}_i\}$ are the solution of \eqref{ClosedLoopProblem3} with $v=\tilde{\nu}$.
\end{prop}
\begin{proof}[Proof of \Cref{ExistenceProp}]
 Consider $\varepsilon\in (0,1)$ and a function $\nu_{\varepsilon}\in C^{2}(\overline{\Omega})$ such that
\beq  \norm{\tilde{\nu}-\nu}_{C^{2}(\overline{\Omega})}\leq \varepsilon \mbox{ and }\nabla \tilde{\nu}(0)=0.\label{ConvergenceV2}\eeq
By \Cref{TimeExistenceLemma} there exists a time $T_{\varepsilon}>0$, such that for each $i\in \{1,\ldots,I\}$, problem \eqref{ClosedLoopProblem3} with $v=\nu^{\varepsilon}$ and $T=T_{\varepsilon}$ admits a solution $y_i^{\varepsilon}\in C^{1}([0,T_{\varepsilon}];\overline{\Omega})$, which satisfies \eqref{TimeExistenceLemma:conv}. Moreover,  we know that $T_{\varepsilon}$ fulfills \eqref{TimeExistenceLemma:Teps}.\par
Due to the exponential stability of \eqref{ExistenceProp:LinearizedSys}, there exists a symmetric and positive definite matrix $M\in \R^{d\times d}$ (see Theorem 4.6 in \cite[p.~136]
{Khalil}), such that
$$ y^{\top}(A^{\top}M+MA)y=-|y|^{2}, \mbox{ for all }y\in \R^{d},$$
where $A=Df(0)-\frac{1}{\beta}BB^{\top}\nabla^{2}\nu(0)$. This equality and the fact that $\nu$ is $C^{2}(\overline{\Omega})$, implies that there exists $\rho >0$, such that
\beq 2(f(y)- \frac{1}{\beta}BB^{\top}\nabla \nu(y))^{\top}My<-\frac{3}{4} |y|^{2}\mbox{ for all }y\in B(0,\rho)\subset\Omega. \label{LyapunovIneq}\eeq
By \eqref{ConvergenceV2} and the integral mean value theorem we have
\beq|\nabla \tilde{\nu}(y)-\nabla \nu(y)|\leq \varepsilon|y|\mbox{ for all }y\in\overline{\Omega}. \label{ConvergenceV2Lip}\eeq
For $\varepsilon<\frac{\beta}{8|B|^{2}|M|}$ and combining \eqref{ConvergenceV2}, \eqref{ConvergenceV2Lip}, and \eqref{LyapunovIneq} we have
\beq 2(f(y)- \frac{1}{\beta}BB^{\top}\nabla \tilde{\nu}(y))^{\top}My<-\frac{1}{2} |y|^{2}\mbox{ for all }y\in B(0,\rho).\label{ExistenceProp:StabVeps}\eeq
Thus, $\psi(y)=y^{\top}My$ is a Lyapunov function for \eqref{ClosedLoopProblem3} with $v=\nu^{\varepsilon}$ in $B(0,\rho)$, \cite[Section 4.4, Theorem 4.10]{Khalil}.

By \eqref{ExistencePropStability}, we know that there exists $T>0$ such that
$$ |y_{i}(t)|<\frac{\rho}{4}\mbox{ for all }t\geq T\mbox{ and }i\in \{1,\ldots,I\}.$$
Further, by \eqref{TimeExistenceLemma:Teps}, we know that $T_{\varepsilon}>T$ if  $\varepsilon<C^{1/2}\sigma/(e^{CT}-1)^{1/2}$. Hence, by \eqref{TimeExistenceLemma:conv} and choosing $\varepsilon$ satisfying $\varepsilon<C^{1/2}\sigma/(e^{CT}-1)^{1/2}$ we have
\beq|y^{\varepsilon}_{i}(T)|\leq |y^{\varepsilon}_{i}(T)-y_i(T)|+ |y_i(T)|\leq\frac{\varepsilon}{C^{1/2}}(e^{CT}-1)^{1/2}+|y_i(T)|.\label{ExistenceProp:proof:yTBound}\eeq
Choosing
$$
\varepsilon\leq \varepsilon_0:=\min\left\{\frac{\rho C^{1/2}}{4(e^{CT}-1)^{1/2}} ,\frac{C^{1/2}\sigma}{(e^{CT}-1)^{1/2}},\frac{\beta}{8|B|^{2}|M|}\right\},
$$
we have that
$\frac{\varepsilon}{C^{1/2}}(e^{CT}-1)^{1/2} \leq \frac{\rho}{4}$ and by  \eqref{ExistenceProp:proof:yTBound} we obtain
$$ |y^{\varepsilon}_{i}(T)| \leq\frac{\rho}{2}  \mbox{ for all }i\in \{1,\ldots,I\}.$$
Given that $\psi$ is a Lyapunov function in $B(0,\rho)$, we have that there exist $K>0$ and $\kappa>0$ such that \eqref{ExistenceProp:ExpStabVesp} holds for all $\tilde{\nu}$ satisfying \eqref{eq:11.41}.
 Further, using that $\tilde \nu$ is $C^{1,1}(\overline{\Omega})$, $\ell$ is $C^{1}$, and \eqref{ExistenceProp:ExpStabVesp}, we get that $\mathcal{J}_{\infty}(\tilde\nu)<\infty$. Therefore, \Cref{ExistenceProp} holds with $\varepsilon_0$.
\end{proof}

\vskip 0.2in
\bibliography{biblio}

\end{document}